\documentclass[11pt]{article}
\usepackage[T1]{fontenc}
\usepackage[utf8]{inputenc}
\usepackage{amsfonts,amssymb,amsmath,mathtools,amsthm}
\usepackage{fmtcount}
\usepackage{url} 
\usepackage{enumerate}
\usepackage[colorlinks=true, linkcolor=blue, citecolor=blue]{hyperref}

\newtheorem{theorem}{Theorem}[section]
\newtheorem{corollary}[theorem]{Corollary}
\newtheorem{lemma}[theorem]{Lemma}
\newtheorem{proposition}[theorem]{Proposition}

\theoremstyle{definition}
\newtheorem{definition}[theorem]{Definition}
\newtheorem{remark}[theorem]{Remark}

\newtheorem{counterexample}[theorem]{Counterexample}

\makeatletter
\newcommand{\pushright}[1]{\ifmeasuring@#1\else\omit\hfill$\displaystyle#1$\fi\ignorespaces}

\newcommand\tq{;\,} 
\newcommand\R{\mathbb R}

\newcommand\abs[1]{\left\vert{#1}\right\vert}
\newcommand\accol[1]{\left\{{#1}\right\}}
\newcommand\CCO[1]{\left({#1}\right)}

\newcommand\norm[1]{\left\Vert{#1}\right\Vert}
\newcommand\un[1]{\,\rlap{{1}}\kern.22em \mbox{l}_{#1}} 
\newcommand\Id{\mathop{\text{\rm Id}}\nolimits}

\newcommand\ellp[1]{\mathop{\text{\rm L}}\nolimits^{\!#1}}
\newcommand\Linear{\mathrm{L}} 
\newcommand\HSz{\mathrm{L}_2^0} %

\newcommand\esprobb{\widetilde{\Omega}}
\newcommand\tribu{\mathcal{F}}

\newcommand\prob{\mathop{\text{\rm P}}\nolimits}
\DeclareMathOperator{\expect}{\mathrm{E}}
\newcommand\law[1]{ \mathop{\mathrm{law}}({#1}) } 
\newcommand\laws[1]{ {\mathcal{M}^{+,1}}\CCO{#1}  }
\newcommand\bor[1]{\mathcal{B}\CCO{#1}}  

\newcommand\ttribu{\widetilde{\tribu}}
\newcommand\oomega{\widetilde{\omega}}
\newcommand\pprob{\mathop{\widetilde{\text{\rm P}}}\nolimits}
\newcommand\ellzz{\ellp{0}(\esprobb;\esp)}

\DeclareMathOperator{\Cb}{\mathrm{C}} 
\DeclareMathOperator{\Cbu}{\mathrm{C}_{\mathrm{u}}} 
\DeclareMathOperator{\Cbk}{\mathrm{C}_{\mathrm{k}}}

\newcommand\point{{\displaystyle{\text .}}} 

\renewcommand\epsilon{\varepsilon}

\newcommand\trace[1]{\mathop{\mathrm{tr}}(#1)}
\newcommand\domain{\mathop{\text{\rm dom}}}

\DeclareMathOperator{\rds}{\varphi}
\DeclareMathOperator{\flow}{\Pi}
\DeclareMathOperator{\shft}{\theta} 
\DeclareMathOperator{\linshft}{\ell} 
\DeclareMathOperator{\shfT}{\vartheta} 
\newcommand{\transl}{\mathfrak{T}} 
\newcommand{\translsimple}[1]{\transl^0_{#1}}
\DeclareMathOperator{\backshift}{\transl}
\newcommand\lebesgue{\lambda}
\newcommand\orbit{{orbit of }}

\DeclareMathOperator{\SDE}{\Gamma}

\newcommand\esp{\mathbb{X}}
\newcommand\espY{\mathbb{Y}}
\newcommand\espZ{\mathbb{Z}}
\newcommand\espB{\mathbb{B}}
\newcommand\traj{\mathbb{Y}} 
\newcommand\dist{\mathfrak{d}}
\newcommand\ellzp[1]{\ellp{#1}(\Omega;\esp)}
\newcommand\ellz{\ellzp{0}}
\newcommand\dizp[1]{\mathop{\mathfrak{d}_{\ellp{#1}}}\!}
\newcommand\diz{\dizp{0}}
\newcommand\wass{\mathop{\mathrm{Wass}_0}}
\newcommand\dizz{\mathop{\mathfrak{d}_{\ellp{0},\infty}}\!}
\newcommand\distCk{\dist_{\mathrm{C}_\mathrm{k}}}
\newcommand\wassCk{\mathop{\mathrm{Wass}_{\mathrm{C}_\mathrm{k}}}}
\newcommand\wassk[1]{\mathop{\mathrm{Wass}_{#1}}}

\newcommand\Js[2]{{#1}^{#2}}
\newcommand\Jse[1]{\Js{#1}{\epsilon}}

\newcommand\APB{\mathrm{AP}}
\newcommand\aperiods[1]{T_{#1}} 
\newcommand\PB{\mathrm{Per}}
\newcommand\CUB{\mathrm{CUB}(\Omega;\espH)}
\newcommand\espH{\mathbb{H}}
\newcommand\U{\mathbb{U}}

\newcommand\ctlip{\mathfrak{h}} 
\newcommand\ctgrowth{\mathfrak{g}} 
\newcommand\ctconv{\mathfrak{c}} 
\newcommand\Ja{{F}^{\tau}}
\newcommand\Jb{{G}^{\tau}}

\newcommand\ddelta{\mathfrak{d}}

\title{Almost periodicity and periodicity
  for nonautonomous random dynamical systems}
\author{Paul \sc{Raynaud de Fitte
  \footnote{Normandie Univ., Laboratoire Rapha\"el Salem,
UMR CNRS 6085, Rouen, France.
E-Mail: prf@univ-rouen.fr
} }
}

\date{}

\allowdisplaybreaks
\numberwithin{equation}{section}
\begin{document}
\maketitle

Keywords : metric dynamical system; 
random dynamical system; crude cocycle; 
nonautonomous dynamical system; stochastic partial differential equation; 
almost periodic

\tableofcontents

\begin{abstract}
  We present a notion of almost periodicity wich can be applied to
  random dynamical systems as well as 
  almost periodic stochastic differential equations in Hilbert spaces
  (abstract stochastic partial differential equations).
  This concept allows for improvements of known results of almost
  periodicity in distribution,
  for general random processes
  and for solutions to stochastic differential
  equations. 
\end{abstract}

\section{Introduction}
Since the introduction of almost periodicity by H.~Bohr in the 1920s
\cite{bohrI,bohrII,bohrIII},
many new definitions and variants of almost periodicity
appeared 
for functions of a real variable (almost periodicity in the sense of
Stepanov, or
Weyl, or Besicovich, almost automorphy...), 
with applications in various fields of mathematics, especially ordinary
differential equations and dynamical systems.
See \cite{andres-bersani} for an overview. 

In the case of random processes, each of these notions forks into
several possible notions, mainly: in distribution (in various senses), in
probability (or in $p$-mean), or in path distribution.
Surveys on such notions in the case of
Bohr almost periodicity can be found in 
\cite{BMRF,Tudor95ap_processes}.

The study of almost periodicity for SDEs seems to start with the Romanian
school, which studied Bohr almost periodicity in one-dimensional
distribution (APOD) of solutions to almost periodic SDEs: 
first Halanay \cite{halanay87}, then  
mostly Constantin Tudor and his collaborators in many papers, among
them  
\cite{arnold-tudor,Morozan-Tudor89,Tudor92affine,Tudor92flows,Tudor98nuclear}.
It was Tudor \cite{Tudor95ap_processes} who proposed the notions of
almost periodicity in finite-dimensional distribution (APFD) and
almost periodicity in (path) distribution, that we call here APPD. 
With G.~Da Prato in \cite{DaPrato-Tudor95}, Tudor proved APPD  for
solutions to semilinear evolution equations in Hilbert spaces.
Many papers continue to appear on almost periodicity of solutions to
various almost periodic SDEs. In most of these papers,
it is the weaker APOD property which is proved,
under the name ``almost periodicity in distribution''. 

In a series of papers which started in 2007, some authors claimed 
the existence of nontrivial
square-mean almost periodic solutions to general semilinear
SDEs in Hilbert spaces. 
Unfortunately, these claims proved to be wrong, even in most elementary
examples such as Ornstein-Uhlenbeck stationary process  
\cite{almost-automorphic,counterexamples}.
The error shared by all these papers was an impossible change of variable in the
Itô integral.
This error is one of the motivations of the present work. Indeed, this
change of variable problem leads naturally to the Wiener shift and
metric dynamical systems.

Recently, W. Zhang and Z.~H. Zheng 
\cite{zhang-zheng19ap} proposed a notion of Bohr almost periodicity for
orbits of random dynamical systems which is the same as ours, in the line
of the notion of periodicity proposed by Zhao and his collaborators
\cite{feng-wu-zhao2016,feng-zhao2012,%
  zhao-zheng2014arXiv,feng-zhao-zhou2011,zhao-zheng2009}. 
Our setting is however less restrictive, since cocycles are only optional, and
the random dynamical systems we consider are nonautonomous and 
they are not necessarily perfect cocycles.

The paper is organized as follows: we present in Section
\ref{sec:setting} our general setting of a probability space endowed
with a group of measure preserving transformations 
(a metric dynamical system), 
and we present the notion of
nonautonomous random dynamical system
that we use in some places: 
it is simply a mixture of nonautonous
dynamical system and very crude cocycle over a metric dynamical
system.
In Section \ref{sec:aprds}, we study a general notion of
almost periodicity, called $\shft$-almost periodicity,
for random processes with values in a Polish space,
in connection with the underlying shift $\shft$
on the probability space.
This notion encompasses almost periodicity in probability and,
under a uniform integrability condition, almost periodicity in $p$-mean. 
In Section \ref{sec:distrib}, we examine the relation
between $\shft$-almost periodicity
and different notions of almost periodicity in
distribution. The strongest one 
is almost periodicity in path
distribution (APPD).
This notion is not implied by $\shft$-almost periodicity
(we exhibit a counterexample),
but we provide 
a sufficient condition,
in the case of continuous processes,
for both $\shft$-almost periodicity and APPD. 
We conclude in Section
\ref{sec:SDE} with an application to stochastic differential
equations in Hilbert spaces,
where we improve and simplify
a general result on existence and uniqueness of
almost periodic solutions, thanks to the metric
dynamical system point of view.
Our method allows us to tackle 
periodicity 
as a particular case.
We show that the unique bounded mild solution to some almost periodic
semilinear
stochastic differential equation in Hilbert spaces is
$\shft$-almost periodic and almost periodic in path distribution.

\section{General setting}\label{sec:setting}
In all the sequel,
$\esp$ is a Polish space, that is, a separable topological
space, whose topology is induced by a metric $\dist$ such that
$(\esp,\dist)$ is complete.

Unless specifically stated, we 
identify random variables which are equal $\prob$-almost everywhere,
and we 
denote by $\ellz$ the space of equivalence classes, for almost
everywhere equality, 
of measurable mappings from $\Omega$ to $\esp$.
This space is
endowed with the distance
\[ \diz\CCO{X,Y}=\expect(\dist(X,Y)\wedge 1),\]
where $\expect$ denotes the expectation with respect to $\prob$. The
distance $\diz$
is complete and compatible with the topology of convergence in
$\prob$-probability.
The law of an element $X$ of $\ellz$ is denoted by $\law{X}$.  The set
of Borel probability measures on $\esp$ is denoted by
$\laws{\esp}$. We endow it with the topology of narrow (or weak)
convergence, that is, the coarsest topology for which the function
\[\left\{ \begin{array}{lcl}
                   \laws{\esp} &\rightarrow&\R\\
                   \mu&\mapsto&\mu(f):=\int_\esp f\,d\mu
          \end{array}\right.\]
      is continuous for every bounded continuous function
      $f :\,\esp\rightarrow\R$.
      The space $\laws{\esp}$ is Polish, see, e.g.,
      \cite{parthasarathy}. 
      A distance which is complete and compatible with the topology of
      $\laws{\esp}$ is the Wasserstein distance $\wass$ associated with
      the truncated metric $\dist\wedge 1$, defined by
      \[
        \wass(\mu,\nu)=\inf_{\law{X}=\mu,\,\law{Y}=\nu}\diz\CCO{X,Y}.\]

\paragraph{Nonautonomous random dynamical systems}
In the sequel, we are given a \emph{metric dynamical system}
$(\Omega,\tribu,\prob,\shft)$, that is,
$(\Omega,\tribu,\prob)$ is a probability space, and
the \emph{shift transformation}
\[\shft :\,\left\{\begin{array}{lcl}
                    \R\times\Omega&\rightarrow&\Omega\\
                    (t,\omega)&\mapsto&\shft_t\omega
\end{array}\right.\]
is $\bor{\R}\otimes\tribu$-measurable, where $\bor{\R}$ is the Borel
$\sigma$-algebra of $\R$, such that $\shft_0=\Id_\Omega$, 
$\shft_{t+s}=\shft_t\circ\shft_s$ for all $s,t\in\R$, and $\prob$ is
invariant under $\shft_t$ for all $t\in\R$ (we shall express this by
saying, for short, that $\prob$ is $\shft$-invariant).
Let $\shfT$ denote the shift transformation on $\R$ 
defined by $\shfT_ts=s+t$ for all
$s,t\in\R$. Then 
\[\Theta :\,\left\{\begin{array}{lcl}
           \R\times\R\times\Omega&\rightarrow&\R\times\Omega\\
           (t,s,\omega)&\mapsto&\Theta_t(s,\omega)=(\shfT_ts,\shft_t\omega)
\end{array}\right.\]
is a flow on $(\R\times\Omega,\bor{\R}\otimes\tribu)$
which preserves the measure $\lebesgue\otimes\prob$, where $\lebesgue$
is the Lebesgue measure on $(\R,\bor{\R})$.

We shall sometimes assume the existence of a
\emph{very crude cocycle}~$\rds$ over 
$(\R\times\Omega,\bor{\R}\otimes\tribu,\lebesgue\otimes\prob,\Theta)$,
more precisely, a measurable mapping
\[\rds :\,\left\{ \begin{array}{lcl}
      \R^+\times\R\times \Omega\times \esp&\rightarrow&\esp\\
       (t,\tau,\omega,x)&\mapsto&\rds(t,\tau,\omega,x)=:\rds(t,\tau,\omega)x
                  \end{array}\right.
\]
such that
$\rds(0,\tau,\omega)=\Id_\esp$ for every $\tau\in\R$  and
$\prob$-almost every $\omega\in\Omega$,
and satisfying the 
\emph{crude cocycle property}
\begin{equation}\label{eq:crude}
  \rds(r+s,\tau,\omega)
  =\rds(r,\Theta_s(\tau,\omega))\circ\rds(s,\tau,\omega)
  =\rds(r,\tau+s,\shft_s\omega)\circ\rds(s,\tau,\omega)
\end{equation}
(with a slight abuse of notations)
for all $(r,s,\tau)\in\R^+\times\R^+\times\R$
and for $\prob$-a.e.~$\omega\in\Omega$.
Note that the almost sure set may depend on $(t,s,\tau)$.

\begin{remark}
\begin{enumerate}
\item The cocycle $\rds$ is called a \emph{measurable random dynamical
  system (measurable RDS)}
\cite{arnold} 
if it is 
independent of the second variable,
that is, if $\rds$ has the form $\rds(t,\tau,\omega)=\rds(t,\omega)$ for all
$(t,\tau,\omega)\in\R^+\times\R\times\Omega$. 
On the other hand, if $\rds$ is 
deterministic (that is, independent of $\omega$),
it is called a \emph{nonautonomous dynamical
  system} (see \cite{caraballo-han,kloeden-rasmussen}). 
Thus nonautonomous random dynamical systems are a combination of these
two notions. 
A similar definition is given in \cite{wang2012}.

\item When the almost sure set in \eqref{eq:crude}
is independent of $(t,s,\tau)$,  the cocycle
$\rds$ is said to be  \emph{perfect}.
Perfection theorems allow to construct perfect modifications of 
crude cocycles, see \cite{arnold-scheutzow} or \cite[Theorem 1.3.2]{arnold}. 
Perfection of the cocycle 
provides powerful tools such as the multiplicative ergodic theorem
(see \cite{arnold}).
However, we are interested here in applications to stochastic
differential equations, possibly in infinite dimensions, see Section
\ref{sec:SDE}. 
For such systems, even with global Lipschitz and growth conditions,
a stochastic flow does not always exist,
see \cite[Section 9.1.2]{dapratozabczyk} or \cite{flandoli},
in that case they are not perfectible. 

\end{enumerate}
\end{remark}

In this paper, the cocycles we consider satisfy only a mild continuity 
assumption:
\begin{definition}
  We say that the cocycle
  $\rds :\,\R^+\times\R\times \Omega\times \esp\rightarrow \esp$
  is {\em continuous in probability}, or
  {\em$\ellz$-continuous}, if
  the mapping
  \[
  \left\{ \begin{array}{lcl}
            \R\times\R\times\esp &\rightarrow&\ellz\\
            (t,\tau,x)&\mapsto&\rds(t,\tau,\point)x.
           \end{array}\right.          
       \]
  is continuous. 
\end{definition}

{\em In the sequel, $\rds$ always denotes a very crude $\ellz$-continuous
  cocycle over
  $(\R\times\Omega,\bor{\R}\otimes\tribu,\lebesgue\otimes\prob,\Theta)$.}

\begin{definition}
  \label{def:grds}
A \emph{(complete) orbit} of $\rds$ is a measurable mapping 
$X :\,\R\times\Omega\rightarrow \esp$ such that,
for all $t,s\in\R$,
\begin{equation}\label{eq:Xrds}
  X(t+s,\point)=\rds(t,s,\shft_s\point)X(s,\point),
\end{equation}
where the equality \eqref{eq:Xrds} holds in $\ellz$,
that is, $\prob$-almost everywhere,
and the almost sure set may depend on $t$ and $s$.
\end{definition}

\begin{proposition}\label{prop:continuousversion}
Every \orbit $\rds$ is continuous in probability. 
\end{proposition}
\begin{proof}
  Let $X$ be an \orbit $\rds$. 
Continuity in probability of $X$ follows from
 \begin{equation*}
  X(t,\point)=\rds(t-t_0,t_0,\shft_{t_0}\point)X(t_0,\point)
\end{equation*}
which shows that $X(t,\point)$ is continuous on $[t_0,+\infty[$
for any $t_0\in\R$.
\end{proof}

\section{$\shft$-almost periodicity and $\shft$-periodicity
  for  
 random processes}\label{sec:aprds}
\subsection{Almost periodicity in metric spaces}

\begin{definition}\label{def:ap}
  \begin{enumerate}[(a)]
\item A set $A\subset\R$ is said to be \emph{relatively dense}
if, for every $\epsilon>0$, there exists $l>0$ such that every interval
of length $l$ has a nonempty intersection with $A$.

\item Let $x :\,\R\rightarrow \esp$ be a continuous function. Let
$\tau>0$. We say that $\tau$ is an \emph{$\epsilon$-almost period} of $x$ if,
for every $t\in\R$, $\dist(x(t),x(t+\tau))\leq\epsilon$. 

\item A continuous function $x :\,\R\rightarrow \esp$ is said to be
  \emph{almost periodic} (in Bohr's sense) if, for each $\epsilon>0$,
  the set of its $\epsilon$-almost periods is relatively dense.

\item Let $\espY$ be a topological space and $\mathcal{K}$ be the set
  of compact 
  subsets of $\espY$.
  For each $K\in\mathcal{K}$, let $\Cbu(K;\esp)$ denote the space
  of continuous functions from $K$ to 
  $\esp$ endowed with the topology of uniform convergence. 
A continuous function $x :\,\R\times\espY\rightarrow \esp$ is said to be
\emph{almost periodic uniformly with respect to compact subsets of $\espY$}
if, for each $K\in\mathcal{K}$, the mapping
\[
\left\{ \begin{array}{lcl}
                   \R &\rightarrow&\Cbu(K;\esp)\\
                   t&\mapsto&x(t,\point)
          \end{array}\right.
\]
is almost periodic.  

\end{enumerate}
\end{definition}

The proof of the following fundamental theorem can be found in
classical textbooks, see, e.g.,
\cite{levitan-zhikov,corduneanu}.
\begin{theorem}[Bochner's criteria]\label{theo:bochner}
  Let $x :\,\R\rightarrow \esp$ be a continuous function. The
  following statements are equivalent:
  \begin{enumerate}[(i)]
  \item $x$ is almost periodic.

  \item The family of translated mappings $t\mapsto x(t+\point)$,
    where $t$ runs over $\R$, is relatively compact in the space
    $\Cbu(\R;\esp)$ of continuous functions from $\R$ to 
  $\esp$ endowed with the topology of uniform convergence. 

\item\label{item:double}
  For every pair of sequences $(\alpha'_n)$ and $(\beta'_n)$ in
  $\R$, there are subsequences $(\alpha_n)$ of $(\alpha'_n)$ and
  $(\beta_n)$ of $(\beta'_n)$ respectively, with same indices, such
  that, for every $t\in \R$, the limits
\begin{equation*}
  \lim_{m\rightarrow\infty}\lim_{n\rightarrow\infty}
         x(t+\alpha_n+\beta_m)
         \text{ and }
   \lim_{n\rightarrow\infty}x(t+\alpha_n+\beta_n) 
\end{equation*}
exist and are equal.
  \end{enumerate}
\end{theorem}

\begin{remark}
 Characterization \eqref{item:double} shows that almost periodicity depends only
on the topology of $\esp$, not on its metric nor on any uniform
structure on $\esp$. 
\end{remark}

Using Characterization \eqref{item:double}, one gets immediately the
following useful result: 

\begin{corollary}[almost periodicity in product spaces]\label{cor:ap-product}
  Let $\espY$, $\espZ$ be metric spaces, and let $x$, $y$, $z$ be
  continuous functions from $\R$ to $\esp$, $\espY$ and $\espZ$
  respectively. The
  following statements are equivalent:
  \begin{enumerate}[(i)]

  \item The functions $x$, $y$ and $z$ are almost periodic. 

    \item The function $t\mapsto(x(t),y(t),z(t))$ with values in
    $\esp\times\espY\times\espZ$ is almost periodic.
    
  \end{enumerate}
\end{corollary}

The definition of almost periodicity
can be extended without change to semimetric spaces, 
and, by passing to a quotient space, one sees that
Theorem \ref{theo:bochner} remains true if $\dist$ is only a semidistance.  
Furthermore, if the topology of $\esp$ is defined by a family
$(\dist_i)_{i\in I}$ of semidistances, we can define almost
periodicity using these semidistances.
The following result (see, e.g., \cite[Lemma 4.4]{BMRF}) will also be
useful in the sequel.

\begin{proposition}[almost periodicity for a family of semidistances]
  \label{prop:semidistances}
  Assume that the topology of $\esp$ is defined by a family
  $(\dist_i)_{i\in I}$ of semidistances.
  For each $i\in I$, let us denote by $(\esp,\dist_i)$ the
space $\esp$ endowed with the (non separated) topology associated with
$\dist_i$. 
  Let $x :\,\R\rightarrow \esp$ be a continuous function. The
  following statements are equivalent:
  \begin{enumerate}[(i)]
  \item $x$ is almost periodic.
  \item For each $i\in I$,
    the function $x :\,\R\rightarrow(\esp,\dist_i)$
    is almost periodic. 
  \end{enumerate}  
\end{proposition}

\subsection{$\shft$-almost periodicity and $\shft$-periodicity}

The following notion of almost periodicity appeared in 
\cite{zhang-zheng19ap}, in the context of continuous random dynamical
systems. It is the natural generalization of the notion of periodicity
investigated by Zhao and his collaborators
\cite{feng-wu-zhao2016,feng-zhao2012,%
  zhao-zheng2014arXiv,feng-zhao-zhou2011,zhao-zheng2009}, 
see also Cherubini et al \cite{cherubini}
for periodicity in
the nonautonomous case. 
 Similarly, the notion of stationarity below (in the autonomous case) 
 can be found in 
\cite{maslowski-schmalfuss04}. 

\begin{definition}[$\shft$-almost periodicity  and $\shft$-periodicity]%
  \label{def:aprds}
  Let $X :\,\R\mapsto \ellz$ be a random process.  
  \begin{enumerate}[(a)]
\item   Let $\epsilon>0$. We say that a number $\tau\in\R$ is
  a \emph{$\shft$-$\epsilon$-almost period of $X$ in probability}
  (or  simply a \emph{$\shft$-$\epsilon$-almost period of $X$})
  if
  \begin{equation}\label{eq:epsperiod}
  \sup_{t\in\R}\diz\CCO{X(t+\tau,\shft_{-\tau}\point),X(t,.)}\leq\epsilon.
\end{equation}
  
  \item \label{cond:ltap}We say that $X$ is 
  \emph{$\shft$-almost periodic in probability} 
  (or simply \emph{$\shft$-almost periodic})  if
  Conditions (i) and (ii) below are satisfied:
  \begin{enumerate}[(i)]
  
  \item\label{cond:cont} 
    the mapping
  \[ \left\{ \begin{array}{lcl}
               \R\times\R&\rightarrow&\ellz\\
               (t,s)&\mapsto& X(t+s,\shft_{-s}\point)
             \end{array}\right.\]
  is continuous in probability,
\item\label{cond:reldense} for any $\epsilon>0$, the
set of $\shft$-$\epsilon$-almost periods of $X$ is relatively dense. 
\end{enumerate}
In the case when $\shft_t=\Id_\Omega$ for all $t$, we say that $X$ is
\emph{almost periodic in probability}.

\item Let $\tau\in\R$. We say that $X$ is 
  \emph{$\shft$-$\tau$-periodic}
  if, for every $t\in\R$,
  \begin{equation*}
    X(t+\tau,\shft_{-\tau}\point)=X(t,.).
  \end{equation*}

\item We say that $X$ is \emph{$\shft$-stationary} if
  $X$ is $\shft$-$\tau$-periodic for every $\tau\in\R$.
\end{enumerate}
\end{definition}

\begin{remark}
\label{rem:def:rds}
  ~
\begin{enumerate}
  \item The notions of $\shft$-$\epsilon$-period, $\shft$-almost periodicity,
    $\shft$-periodicity and $\shft$-stationarity
    do not depend on a cocycle $\rds$,
    they depend only on the underlying shift $\shft$. 
However, the association of $\shft$-almost periodicity with Property
\eqref{eq:Xrds} will prove useful,
in particular in applications to stochastic differential equations,
see Section \ref{sec:SDE}.

  \item \label{rem:choice_of_l}%
  We can generalize Definition \ref{def:aprds} by replacing
  $\shft_{-\tau}$ by $\shft_{\linshft(\tau)}$ in \eqref{eq:epsperiod},
  for some fixed linear mapping $\linshft :\,\R\rightarrow\R$.
  The reader can check that all results of this section
  and of Section \ref{sec:distrib}
  remain valid
  for any other choice of $\linshft$ than $\linshft(\tau)=-\tau$.
  Actually, each choice of $\linshft$ amounts to a change of the
  metric dynamical system $(\omega,\tribu,\prob, \shft)$ by replacing
  $\shft$ with $\shft'$ given by $\shft'_t=\shft_{-\linshft(t)}$. 
  Another way to see this equivalence is to notice that
  $\shft$-almost periodicity of $X$ amounts to almost periodicity
  in probability
  of $\widetilde{X} :\, t\mapsto(X(t,\shft_t\point))$ (or more
  generally $t\mapsto(X(t,\shft_{\linshft(t)}\point))$).
  This shows also that
  $\shft$-almost periodicity of $X$ can be interpreted as 
  ordinary almost periodicity in the sense of 
  Definition \ref{def:ap} of some function $\widetilde{X}$
  with values in the metric
 space $\ellz$ (see also Proposition \ref{prop:aptransl}).

  However, for orbits of a given cocycle $\rds$,
  since $\linshft$ 
is not taken into account in
\eqref{eq:crude},
each choice of $\linshft$ gives rise to a different class of almost
periodic, periodic or stationary orbits.

For applications to stochastic differential equations,
the choice $\linshft(\tau)=-\tau$ appears to be more relevant,
see Section \ref{sec:SDE}.
The choice of $\linshft=0$ (almost periodicity in probability)
led to wrongful claims in many papers,
see \cite{almost-automorphic,counterexamples} for details.

\item Is is obvious that every $\shft$-stationary random process is
  strictly stationary. 
  Conversely, for every every strictly stationary random process, its
  canical process is a $\shft$-stationary process.
  More precisely, let  $X :\,\R\mapsto \ellz$ be stricly
  stationary. Let $\esprobb=\esp^{\R}$, endowed with the
  $\sigma$-algebra $\ttribu$ generated by cylinder sets, and let  
  $\pprob$ be the law of $X$ on $(\esprobb,\ttribu)$. 
  Set $\widetilde{X}(t,\oomega)=\oomega(t)$ for $\oomega\in\esprobb$
  and $t\in\R$.
  The mapping $\widetilde{X}$ 
  is a version of a random process $\R\mapsto \ellzz$ 
  which is  $\shft$-stationary for the
  shift transformation $\shft$ on $\esprobb$ defined by
  \begin{equation*}
    \shft_t(\oomega)=\oomega(t+\point),\quad \oomega\in\esprobb,\ t\in\R,
  \end{equation*}
 see \cite[Chapter IV]{rozanov} fore more details. 
\end{enumerate}
\end{remark}

\begin{proposition}[{Closure property}]\label{prop:closure}
 Let $(X_n)$ be a sequence of $\shft$-almost periodic random
 processes.  
Assume further that there exists a random process $X$ such that
  \begin{equation*}
    \lim_{n\rightarrow\infty}\sup_{t\in\R}
    \diz\CCO{X_n(t,\point),X(t,\point)}=0. 
  \end{equation*}
  Then $X$ is $\shft$-almost periodic. 
\end{proposition}
\begin{proof}
  Let $\epsilon>0$, and let $N$ such that
  \begin{equation}\label{eq:N}
    \sup_{t\in\R}
    \diz(X_N(t,\point),X(t,\point))\leq\frac{\epsilon}{3}. 
  \end{equation}
  Let $\tau$ be an $\epsilon/3$-period of $X_N$.
  We have, for every $t\in\R$,
  \begin{align*}
    \diz\CCO{X(t+\tau,\shft_{-\tau}\point),X(t,\point)}
    \leq&\diz\CCO{X(t+\tau,\shft_{-\tau}\point),X_N(t+\tau,\shft_{-\tau}\point)}\\
        &+\diz\CCO{X_N(t+\tau,\shft_{-\tau}\point),X_N(t,\point)}\\
        &+\diz\CCO{X_N(t,\point),X(t,\point)}\\
    \leq&\epsilon.
  \end{align*}
  We deduce that the set of $\epsilon$-almost periods of $X$ is
  relatively dense.

  To prove continuity in probability of
  $(t,s)\mapsto X(t+s,\shft_{-s}\point)$,
  let $t_0,s_0\in\R$, and 
  choose again $N$ satisfying \eqref{eq:N}.
  Let $\eta>0$ such that
  \[\max\{\abs{t-t_0},\abs{s-s_0}\}<\eta\Rightarrow
 \diz\CCO{X_N(t+s,\shft_{-s}),X_N(t_0+s_0,\shft_{-s_0})}\leq\frac{\epsilon}{3}.\]
  We have, using the invariance of $\prob$ by $\shft$,
  \begin{align*}
    \diz\CCO{X(t+s,\shft_{-s}),X(t_0+s_0,\shft_{-s_0})}
    \leq &\diz\CCO{X(t+s,\shft_{-s}),X_N(t+s,\shft_{-s})}\\
         &+\diz\CCO{X_N(t+s,\shft_{-s}),X_N(t_0+s_0,\shft_{-s_0})}\\
         &+\diz\CCO{X_N(t_0+s_0,\shft_{-s_0}),X(t_0+s_0,\shft_{-s_0})}\\
    \leq& \epsilon.
  \end{align*}
 \end{proof}

We focus now on continuity and compactness properties. If $X$ is an
orbit of $\rds$, Condition \eqref{cond:ltap}-(i)
of Definition \ref{def:aprds} can be decomposed into simpler conditions. 
\begin{proposition}[Joint continuity in probability for orbits of $\rds$]\label{cor:jointcont}
  Let $X$ be an orbit of a very crude $\ellz$-continuous cocycle
  $\rds$. Assume that
  \begin{enumerate}[(i)]
  \item\label{hyp:reldense}
    for every $\epsilon>0$, the set of $\shft$-$\epsilon$-almost
    periods of $X$ is relatively dense,

  \item\label{hyp:contshft} the mapping 
 \[ \left\{ \begin{array}{lcl}
               \R&\rightarrow&\ellz\\
               s&\mapsto& X(s,\shft_{-s}\point)
             \end{array}\right.\]
         is continuous in probability.

  \end{enumerate}
  Then $X$ satisfies 
Property
 \eqref{cond:ltap}-(i) of Definition \ref{def:aprds}. 
\end{proposition}
\begin{proof}
Note that, by \eqref{eq:Xrds},
Hypothesis (\ref{hyp:contshft}) implies that, for every $t\geq 0$, the mapping
$s\mapsto X(t+s,\shft_{-s}\point)$ is continuous in probability.
For $t<0$, we arrive at the same conclusion
with the help of (\ref{hyp:reldense}): let $\epsilon>0$, and let
$\tau$ be an $\epsilon$-almost period of $X$ such that $t+\tau>0$.
We have, for $s_0,s\in\R$
\begin{align*}
  \diz\Bigl(X(t+s,\shft_{-s}),&X(t+s_0,\shft_{-s_0})\bigr)\\
\leq&
  \diz\Bigl(X(t+s,\shft_{-s}),X(t+s+\tau,\shft_{-s-\tau})\bigr)\\
  &+\diz\Bigl(X(t+s+\tau,\shft_{-s-\tau}),X(t+s_0+\tau,\shft_{-s_0-\tau})\bigr)\\
  &+\diz\Bigl(X(t+s_0+\tau,\shft_{-s_0-\tau}),X(t+s_0,\shft_{-s_0 }))\bigr)\\
  \leq&\diz\Bigl(X(t+s+\tau,\shft_{-s-\tau}),X(t+s_0+\tau,\shft_{-s_0-\tau})\bigr)
        +2\epsilon.
\end{align*}
The claim follows from the continuity of
$s\mapsto X(t+\tau+s,\shft_{-s-\tau}\point)$. 

  Let us now prove that the mapping
  $t\mapsto X(t+s,\shft_{-s}\point)$
  is continuous in probability, uniformly with respect to $s$ in
  compact intervals.
  Let $t_0\in\R$, let $J$ be a compact interval, and let $\epsilon>0$.
By continuity of $s\mapsto X(t_0+s,\shft_{-s}\point)$, 
  the family of random elements
  of the form $X(t_0+s,\shft_{-s}\point)$ with $s\in J$ is uniformly
  tight, thus there exists a compact subset $K$ of $\esp$ such that,
  for each $s\in J$, $\prob\{X(t_0+s,\shft_{-s}\point)\not\in K\}\leq
  1-\epsilon/2$.

On the other hand, by uniform continuity in probability of
$(r,s,x)\mapsto\rds(r,t_0+s,\shft_{t_0}\point)x$
on the compact set $[0,1]\times (J+t_0+[-1,1])\times K$,  
and since
$\diz\bigl(\rds(0,t_0+s,\shft_{t_0}\point)x,x\bigr)=0$,
we can find $\delta>0$ such that 
\begin{equation}\label{eq:dephi}
  {0\leq r\leq\delta}\Rightarrow
  \diz\bigl(\rds(r,t+s,\point)x,x\bigr)\leq
  \epsilon/2\quad (t-t_0\in[-1,1],\, s\in J,\, x\in K).
\end{equation}
From \eqref{eq:dephi} and the definition of $\diz$,
we deduce that, for $s\in J$ and
$0\leq t-t_0\leq\delta$, we have
\begin{multline*}
\diz\bigl(X(t+s,\shft_{-s}\point),X(t_0+s,\shft_{-s}\point)\bigr)\\
=\diz\bigl(
\rds(t-t_0,t_0+s,\shft_{t_0}\point)X(t_0+s,\shft_{-s}\point),
                           X(t_0+s,\shft_{-s}\point)\bigr)
\leq \epsilon.
\end{multline*}

If $t\leq t_0$, using again \eqref{eq:dephi}
 and the definition of $\diz$, we get, since $\shft$ is
measure preserving,
\begin{multline*}
  \diz\bigl(X(t+s,\shft_{-s}\point),X(t_0+s,\shft_{-s}\point)\bigr)\\
\begin{aligned}
  =&\diz\bigl(X(t+s,\shft_{-(t+s)}\point),X(t_0+s,\shft_{-(t+s)}\point)\bigr)\\
=&\diz\bigl(
X(t+s,\shft_{-(t+s)}\point)\bigr),
\rds(t_0-t,t+s,\shft_{t}\point)X(t+s,\shft_{-(t+s)}\point)\bigr)
\leq \epsilon.
\end{aligned}
\end{multline*}
So, we have proved our second claim. 

Now, let $t_0,s_0\in\R$, and let
$\epsilon>0$.
By our hypothesis, 
there exists $\delta_1>0$ such that
 $\abs{s-s_0}\leq\delta_1$ implies
\begin{equation}\label{eq:s}
\diz\bigl(X(t_0+s,\shft_{-s}\point),X(t_0+s_0,\shft_{-s_0}\point)\bigr)
\leq \epsilon/2.
\end{equation}
But we have proved that there exists $\delta_2>0$ such that,
for $\abs{t-t_0}\leq\delta_2$, and 
for all
$s\in[s_0-\delta_1,s_0+\delta_1]$, 
\begin{equation}\label{eq:t}
\diz\bigl(X(t+s,\shft_{-s}\point),X(t_0+s,\shft_{-s}\point)\bigr)
\leq \epsilon/2.
\end{equation}
The result follows immediately from \eqref{eq:s} and \eqref{eq:t}.
\end{proof}

\begin{proposition}[Compactness]\label{prop:tightness}
  Let 
$X:\,\R\rightarrow\ellz$ be a $\shft$-almost periodic random process, and let
$J$ be a compact interval of $\R$. 
  Then
  \begin{enumerate}[(i)]
  \item\label{item:l0compact} The set
    $\mathcal{L}_J
    =\{{X(s+t,\shft_{-t}\point)}\tq s\in J,\ {{t}\in\R}\}$ is relatively
    compact in $\ellz$, 
\item\label{item:Xtight}   the set
 $\mathcal{K}=\{\law{X(t,\point)}\tq {t\in\R}\}$ is uniformly tight, that is,
 for each $\epsilon>0$,
 there exists a compact subset $K$ of $\esp$ such that
 \[\sup_{t\in\R}\prob{\{\omega\in\Omega\tq X(t,\omega)\not\in K\}}\leq\epsilon.\]
 \end{enumerate}
\end{proposition}
\begin{proof}
  Let $\epsilon>0$, and let $l>0$ such that any interval of length
  $l$ contains an $\epsilon$-almost period of $X$.
  Let $I=[-l/2,l/2]$. By continuity in
  probability of  $(s,t)\mapsto X(s+t,\shft_{-t}\point)$, the set of
  random variables 
  $\mathcal{J}=\{X(s+t,\shft_{-t}\point)\tq s\in J,\ {t}\in I\}$ is a compact
  subset of $\ellz$.
Now, let $s\in J$, ${t}\in\R$, and let $\tau\in[-{t}-l/2,-{t}+l/2]$
  be an $\epsilon$-almost period of $X$. We have ${t}+\tau\in I$ and
  \[
    \diz\CCO{X(s+t,\shft_{-t}\point),X(s+t+\tau,\shft_{{-t}-\tau}\point)}
    \leq \epsilon,\]
with $X(s+t+\tau,\shft_{{-t}-\tau}\point)\in \mathcal{J}$. 
Thus $\mathcal{J}$ is an $\epsilon$-net of $\mathcal{L}_J$ for the
distance $\diz$, which proves \eqref{item:l0compact}. 

Relative compactness of $\mathcal{K}$ follows
from \eqref{item:l0compact} with $J=\{0\}$
by continuity of the mapping $Y\mapsto\law{Y}$ from $\ellz$ to
$\laws{\esp}$, because, since $\shft_{-t}$ is measure preserving, we have
\[\mathcal{K}=\{\law{X(0+t,\shft_{-t}\point)}\tq t\in\R\}
             =\{\law{Y}\tq Y\in\mathcal{L}_0\}.\]
  Since $\esp$ is Polish, $\mathcal{K}$ is uniformly tight, by
  Prokhorov's well-known theorem (see, e.g.,
  \cite{bogachev}), which proves \eqref{item:Xtight}.
\end{proof}

\begin{theorem}[Equicontinuity
    and uniform continuity in probability]\label{theo:jointunifcont}
    Let
$X:\,\R\rightarrow\ellz$ be a $\shft$-almost periodic random process. 
    Then,
  \begin{enumerate}[(a)]
  \item The mapping
  $t\mapsto X(t+s,\shft_{-s}\point)$
  is continuous in probability, uniformly with respect to $s\in\R$.

  \item \label{item:unifcont} The mapping 
  $s\mapsto X(t+s,\shft_{-s}\point)$
  is uniformly continuous in probability,
  uniformly with respect to $t\in\R$.
  \end{enumerate}
\end{theorem}
\begin{proof}
  (a) Joint continuity in probability of $(t,s)\mapsto
  X(t+s,\shft_{-s}\point)$ is provided by 
Condition \eqref{cond:ltap}-(i)
of Definition \ref{def:aprds}. 
  For the uniformity with respect to $s\in\R$, 
  let $\epsilon>0$, and let $l>0$ such that any interval of length
  $l$ contains an $\epsilon/3$-almost period of $X$.
  For each relative integer $k$, set $I_k=[-l/2+kl,l/2+kl]$,
  and let $J=[-l,l]$.
  Let $t_0\in\R$.
  By uniform continuity of $(t,s)\mapsto X(t+s,\shft_{-s}\point)$
  on $[t_0-1,t_0+1]\times J$, 
  there exists $\delta\in]0,1]$ such that,
  for all $s\in J$  and for $\abs{t-t_0}\leq\delta$, we have
  \begin{equation*}
    \diz\Bigl(X(t+s,\shft_{-s}\point),
              X(t_0+s,\shft_{-s}\point)\Bigr)\leq\epsilon/3. 
   \end{equation*}
   Let $s\in\R$, let $k$ such that $s\in I_k$, and let $\tau_{-k}\in
   I_{-k}$ be an $\epsilon/3$-almost period of $X$. We have $s+\tau_k\in
   J$, thus
   \begin{multline*}
     \diz\Bigl(X(t+s,\shft_{-s}\point),
     X(t_0+s,\shft_{-s}\point)\Bigr)\\
     \begin{aligned}
       \leq & \diz\Bigl(X(t+s,\shft_{-s}\point),
               X(t+s+\tau_k,\shft_{-s-\tau_k}\point)\Bigr)\\
              &+\diz\Bigl(X(t+s+\tau_k,\shft_{-s-\tau_k}\point),
              X(t_0+s+\tau_k,\shft_{-s-\tau_k}\point)\Bigr)\\
              &+\diz\Bigl(X(t_0+s+\tau_k,\shft_{-s-\tau_k}\point),
              X(t_0+s,\shft_{-s}\point)\Bigr)\\
         \leq
         &\frac{\epsilon}{3}+\frac{\epsilon}{3}+\frac{\epsilon}{3}=\epsilon. 
     \end{aligned}
   \end{multline*}

(b)  Let $t_0\in\R$. 
  Let $\epsilon>0$, and let $l>0$ such that any interval of length $l$
  contains an $\epsilon/3$-almost period of $X$. Let $J=[-l,l]$. The mapping
   $s\mapsto X(t_0+s,\shft_{-s}\point)$ is uniformly continuous in
   probability on 
   $J$, thus we can find $\delta>0$ such that
   $\diz\CCO{X(t_0+s,\shft_{-s}\point),X(t_0+r,\shft_{-r}\point)}<\epsilon/3$
   for all $r,s\in J$ such
   that $\abs{s-r}<\delta$. We can choose $\delta<l$.
   Let $r,s\in\R$ with $\abs{s-r}<\delta$, and let $m=(r+s)/2$, so
   that $r,s\in[m-l/2,m+l/2]$. 
   Let $\tau\in[-m-l/2,-m+l/2]$ be an $\epsilon/3$-almost period of
   $X$. We have 
   \begin{multline*}
     \diz\CCO{X(t_0+r,\shft_{-r}\point),X(t_0+s,\shft_{-s}\point)}\\
     \begin{aligned}
 \leq &\diz\CCO{X(t_0+r,\shft_{-r}\point),X(t_0+r+\tau,\shft_{-r-\tau}\point)}\\
 & +\diz\CCO{X(t_0+r+\tau,\shft_{-r-\tau}\point),
          X(t_0+s+\tau,\shft_{-s-\tau}\point)}\\
        & +\diz\CCO{X(t_0+s+\tau,\shft_{-s-\tau}\point),
          X(t_0+s,\shft_{-s}\point)}\\
     \leq &\frac{\epsilon}{3}+\frac{\epsilon}{3}+\frac{\epsilon}{3}=\epsilon.
     \end{aligned}
   \end{multline*}
Then, for any $t\in\R$, and for $r,s\in\R$ such that
$\abs{s-r}<\delta$, we get, since $\shft_{t-t_0}$ is measure preserving,
 \begin{multline*}
     \diz\CCO{X(t+r,\shft_{-r}\point),X(t+s,\shft_{-s}\point)}\\
     \begin{aligned}
       =&\diz\CCO{X(t_0+(r+t-t_0),\shft_{-r-t+t_0}\point),
         X(t_0+(s+t-t_0),\shft_{-s-t+t_0}\point)}\\
       \leq&\epsilon.
    \end{aligned}
   \end{multline*}
 \end{proof}

\begin{theorem}[Bochner's criteria for $\shft$-almost periodicity]\label{theo:equivalencesapRDS}
  Let $X :\,\R\mapsto \ellz$
    be a random process 
  satisfying Property \eqref{cond:ltap}-(i) 
  of Definition \ref{def:aprds}. 
  The following statements are 
  equivalent:
  \begin{enumerate}[(i)]
  \item\label{item:ap} $X$ is $\shft$-almost periodic.
    
\item\label{item:normal}
  (Bochner's criterion \cite{Bochner27Beitrage-I})
  The family of mappings 
\[ \transl_sX :\,     
  \left\{ \begin{array}{lcl}
  \R&\rightarrow&\ellz\\
  t&\mapsto &X(t+s,\shft_{-s}\point),
          \end{array}\right.\]
      where $s$ runs over $\R$, 
is relatively compact in
  the space $\Cbu(\R;\ellz)$ of continuous functions from $\R$ to
  $\ellz$ endowed with the topology of uniform convergence,
  that is,  
  for every sequence $(\gamma'_n)$ of real numbers, there exists
  a subsequence $(\gamma_n)$ of $(\gamma'_n)$
  and a random process $Y :\,\R\rightarrow\ellz$,
  continuous in probability,
  such that
  \[\lim_{n\rightarrow\infty}\sup_{t\in\R}
    \diz\CCO{\transl_{\gamma_n}{X}(t,\point),Y(t,\point)}=0.\]

\item\label{item:doubleseq} (Bochner's double sequence criterion
  \cite{bochner62new_approach})
  For every pair of sequences $(\alpha'_n)$ and $(\beta'_n)$ in
  $\R$, there are subsequences $(\alpha_n)$ of $(\alpha'_n)$ and
  $(\beta_n)$ of $(\beta'_n)$ respectively, with same indices, such
  that, for every $t\in \R$, the limits in probability
\begin{equation}\label{def:Bochner double sequence}
  \lim_{m\rightarrow\infty}\lim_{n\rightarrow\infty}
         X(t+\alpha_n+\beta_m,\shft_{-\alpha_n-\beta_m}\point)
         \text{ and }
   \lim_{n\rightarrow\infty}X(t+\alpha_n+\beta_n,\shft_{-\alpha_n-\beta_n}\point)
\end{equation}
exist and are equal.
  \end{enumerate}
\end{theorem}
\begin{proof}

\eqref{item:ap}$\Rightarrow$\eqref{item:doubleseq}:
Let  $(\alpha'_n)$ and $(\beta'_n)$ be two sequences in $\R$, and let
$t\in\R$.
Let $\epsilon>0$, 
let $l$ be such that each interval of length $l$ contains an
$\epsilon/2$-almost period of $X$, and let $J=[-l/2,l/2]$. 
For each $s\in\R$, let $\tau_s\in[s-l/2,s+l/2]$
be an $\epsilon/2$-almost period of
$X$, and let $\Jse{s}=s-\tau_s\in J$.
We can extract from the sequences $(\Jse{\alpha'_n})$ and $(\Jse{\beta'_n})$ 
two subsequences $(\Jse{\alpha}_n)$ and $(\Jse{\beta}_n)$
respectively, with same indices,
which converge in $J$ to some limits $\alpha^\epsilon$ and $\beta^\epsilon$
respectively.
Then, by the continuity property
\eqref{cond:ltap}-(i) 
of Definition \ref{def:aprds},
the following limits in probability exist for any $t\in\R$:
\begin{equation}\label{eq:doublelim}
\begin{aligned}
  \lim_{m\rightarrow\infty}\lim_{n\rightarrow\infty}
  X(t+\Jse{\alpha}_n+\Jse{\beta}_m,\shft_{-\Jse{\alpha}_n-\Jse{\beta}_m}\point)
  &=X(t+\Jse{\alpha}+\Jse{\beta},\shft_{-\Jse{\alpha}-\Jse{\beta}}\point)\\
  &=\lim_{n\rightarrow\infty}
  X(t+\Jse{\alpha}_n+\Jse{\beta}_n,\shft_{-\Jse{\alpha}_n-\Jse{\beta}_n}\point),
\end{aligned}
\end{equation}
whith, for all integers $n,m$
\begin{multline}\label{eq:approxdouble}
  \diz\Bigl(
  X(t+{\alpha}_n+{\beta}_m,\shft_{-{\alpha}_n-{\beta}_m}\point),
  X(t+\Jse{\alpha}_n+\Jse{\beta}_m,\shft_{-\Jse{\alpha}_n-\Jse{\beta}_m}\point)
  \Bigr)\\  
\begin{aligned}
 \leq& \diz\Bigl(
  X(t+{\alpha}_n+{\beta}_m,\shft_{-{\alpha}_n-{\beta}_m}\point),
  X(t+\Jse{\alpha}_n+{\beta}_m,\shft_{-\Jse{\alpha}_n-{\beta}_m}\point)
  \Bigr)\\
  &+ \diz\Bigl(
  X(t+\Jse{\alpha}_n+{\beta}_m,\shft_{-\Jse{\alpha}_n-{\beta}_m}\point),
  X(t+\Jse{\alpha}_n+\Jse{\beta}_m,\shft_{-\Jse{\alpha}_n-\Jse{\beta}_m}\point)
    \Bigr)\\ 
  \end{aligned}\\
  \leq \epsilon/2+\epsilon/2=\epsilon.
\end{multline}
Let us repeat this procedure for $\epsilon=1/k$, where $k\geq 1$ is an
integer, in such a way that, for each $k$, the sequences
$(\Js{\alpha}{1/{(k+1)}}_{n})$ and $(\Js{\beta}{1/{(k+1)}}_{n})$
are subsequences of
$(\Js{\alpha}{1/{k}}_{n})$ and $(\Js{\beta}{1/{k}}_{n})$
respectively. 
Let $(\alpha_n)$ 
and $(\beta_n)$  
be the subsequences of $(\alpha'_n)$ and $(\beta'_n)$ respectively
corresponding to 
$(\Js{\alpha}{1/n}_{n})$ and $(\Js{\beta}{1/n}_{n})$. 
By \eqref{eq:doublelim}, for any integer $N\geq 1$,
the following limits in probability exist:
\begin{multline}\label{def:Bochner double sequenceN}
  \lim_{m\rightarrow\infty}\lim_{n\rightarrow\infty}
  X(t+\Js{\alpha}{1/N}_n+\Js{\beta}{1/N}_m,
          \shft_{-\Js{\alpha}{1/N}_n-\Js{\beta}{1/N}_m}\point)\\
         =
         \lim_{n\rightarrow\infty}X(t+\Js{\alpha}{1/N}_n+\Js{\beta}{1/N}_n,
         \shft_{-\Js{\alpha}{1/N}_n-\Js{\beta}{1/N}_n}\point).
\end{multline}
On the other hand, 
we deduce from \eqref{eq:approxdouble} that,
for $n,m\geq N$,
\begin{equation}\label{eq:approxdoubleN}
  \diz\Bigl(
  X(t+{\alpha}_n+{\beta}_m,\shft_{-{\alpha}_n-{\beta}_m}\point),
  X(t+\Js{\alpha}{1/n}_n+\Js{\beta}{1/n}_m,
                  \shft_{-\Js{\alpha}{1/n}_n-\Js{\beta}{1/n}_m}\point)
  \Bigr)
  \leq \frac{1}{N}.
\end{equation}
The result follows from \eqref{def:Bochner double sequenceN} and 
\eqref{eq:approxdoubleN}, since $N$ is arbitrary.

\eqref{item:doubleseq}$\Rightarrow$\eqref{item:normal}:
Taking $\beta'_n=0$, we see that, for each sequence $(\alpha'_n)$ in
$\R$, there exists a subsequence $(\alpha_n)$ and a random process
$X :\,\R\rightarrow\ellz$ such that, for every $t\in\R$,
$X(t+\alpha_n,\shft_{-\alpha_n}\point)$ converges in probability to
$Y(t,\point)$. So, we only need to prove that this convergence is
uniform with respect to $t$. 
Assuming this is not the case, we get the existence of $\epsilon>0$
and a sequence $(\beta_n)$ in $\R$ such that, for every $n$,
\begin{equation*}
  \diz\Bigl(
  X(\beta_n+\alpha_n,\shft_{-\alpha_n}\point), Y(\beta_n,\point)
  \Bigr)\geq\epsilon,   
\end{equation*}
that is,
\begin{equation}\label{eq:geqepsilon}
  \diz\Bigl(
  X(\beta_n+\alpha_n,\shft_{-\alpha_n-\beta_n}\point),
           Y(\beta_n,\shft_{-\beta_n}\point)
  \Bigr)\geq\epsilon,   
\end{equation}
  By \eqref{item:doubleseq}, extracting subsequences, we may assume
  that 
  we have the following limits in probability:
  \begin{equation*}
 \begin{aligned}
    \lim_{m\rightarrow\infty}Y(\beta_m,\shft_{-\beta_m}\point)
   =& \lim_{m\rightarrow\infty}\lim_{n\rightarrow\infty}
  X(\alpha_n+\beta_m,
  \shft_{-{\alpha}_n-{\beta}_m}\point)\\
         =&
         \lim_{n\rightarrow\infty}X({\alpha}_n+{\beta}_n,
         \shft_{-{\alpha}_n-{\beta}_n}\point)
         \\
         =&
   \lim_{n\rightarrow\infty}\lim_{m\rightarrow\infty}
  X(\alpha_n+\beta_m,
  \shft_{-{\alpha}_n-{\beta}_m}\point), 
  \end{aligned}
\end{equation*}
which contradicts \eqref{eq:geqepsilon}.

\eqref{item:normal}$\Rightarrow$\eqref{item:ap}:
Let $\epsilon>0$. 
By total boundedness of the family $\{\transl_s{X}\tq s\in\R\}$, we
can find a finite sequence $\gamma_1,\dots,\gamma_n$ such that, for
each $s\in\R$, there exists $k\in\{1,\dots,n\}$ such that
\[
  \sup_{t\in\R}\diz\Bigl(
  \transl_s{X}(t),\transl_{\gamma_k}{X}(t)\Bigr)\leq\epsilon,\]
that is,
\begin{equation*}
  \sup_{t\in\R}\diz\Bigl(
  X(t,\point),X(t+s-\gamma_k,\shft_{-s+\gamma_k}),X(t,\point)
  \Bigr)\leq\epsilon,
\end{equation*}
which shows that $s-\gamma_k$ is an $\epsilon$-almost period of
$X$. Let $l=\max\{\gamma_1,\dots,\gamma_n\}$. Then
$s-\gamma_k\in[s-l,s+l]$, thus each interval of length $2l$ contains
an $\epsilon$-almost period of $X$. 
\end{proof}

\begin{remark}\label{rem:topological}
  Bochner's double sequence criterion
  and the proof of Theorem \ref{theo:equivalencesapRDS}
  show that $\shft$-almost periodicity is a property which remains unchanged
  if we replace $\dist$ by any other distance compatible with the
  topology of $\esp$,
  or if we replace $\diz$ by any other distance compatible with
  the topology of convergence in probability.
\end{remark}

\subsection{$\shft$-almost periodicity in $p$-mean}
We present here a stronger notion of almost periodicity,  
which depends on the distance $\dist$ on $\esp$. 

Let $p\geq 1$. Let us denote by $\ellzp{p}$ the set of elements $X$ of
$\ellz$ such that, for some (equivalently, for any) $x_0\in\esp$,
\[\expect\CCO{\dist(X,x_0)}^p<\infty.\]
We endow $\ellzp{p}$ with the distance
\[\dizp{p}(X,Y)=\Bigl(\expect\CCO{\dist(X,Y)}^p\Bigr)^{1/p}.\]
Similarly to definition \ref{def:aprds}, replacing $\diz$ by
$\dizp{p}$, we set:
\begin{definition}\label{def:aprdsp}
  Let $X :\,\R\mapsto \ellz$ be a random process
  such that $X(t,.)\in\ellzp{p}$ for
  each $t\in\R$.
  \begin{enumerate}[(a)]
  \item We say that a number $\tau\in\R$ is
  a \emph{$\shft$-$\epsilon$-almost period of $X$
  in $p$-mean} 
  if
  \[\sup_{t\in\R}\dizp{p}
    \CCO{X(t+\tau,\shft_{-\tau}\point),X(t,.)}\leq\epsilon.\]
  \item We say that $X$ is 
  \emph{$\shft$-almost periodic in $p$-mean} if
  the mapping $(t,s)\mapsto X(t+s,\shft_{-s}\point)$ is continuous
  for the distance $\dizp{p}$ and if, for each $\epsilon>0$,  
  the set of $\shft$-$\epsilon$-almost periods
  in $p$-mean of $X$
  is relatively dense.

In the case when $\shft_t=\Id_\Omega$ for all $t$, we say that $X$ is
\emph{almost periodic in $p$-mean}. 
\end{enumerate}
\end{definition}

It is not difficult to check that all preceding results remain true if
we replace $\diz$ by $\dizp{p}$. Furthermore, using Vitali's theorem
and Bochner's double sequence criterion
in Theorem \ref{theo:equivalencesapRDS}-\eqref{item:doubleseq}, we
have immediately:

 \begin{proposition}\label{prop:pUI}
   Let $p\geq 1$, and
   let $X :\,\R\mapsto \ellz$ be a random process
   such that $X(t,.)\in\ellzp{p}$ for
  each $t\in\R$.   The following statements are 
  equivalent:
  \begin{enumerate}[(i)]
  \item $X$ is $\shft$-almost periodic in $p$-mean.

  \item $X$ is $\shft$-almost periodic and, for some
    (equivalently, for any) $x_0\in\esp$, the random variables
$(\dist(X(t,.),x_0))^p$, $t\in\R$, are uniformly integrable. 
  \end{enumerate}
  \end{proposition}

  \subsection{Almost periodicity in $\Cbu(\R;\ellz)$
    of the translation  map}
\begin{definition}\label{def:translate}
  Let $X :\,\R\mapsto \ellz$
 be a random process.
  Assume that $X$ is continuous in probability.  
  \begin{enumerate}[(a)]
    \item
The \emph{translation operator} of $X$ is the map 
\begin{equation*}
  \transl :\,
  \left\{ \begin{array}{lcl}
              \R&\rightarrow&\Cb(\R;\ellz) \\
              t&\mapsto&\transl_tX(\point,\point)=X(t+\point,\shft_{-t}\point).
   \end{array}\right.      
\end{equation*}

\item
The process $\transl_tX$ is called the \emph{$t$-translate} of $X$.
\end{enumerate}
\end{definition}
We denote by $\Cbu(\R;\ellz)$ the space $\Cbu(\R;\ellz)$
of continuous functions from $\R$ to $\ellz$ 
endowed with the topology of uniform convergence
associated with $\diz$.
We denote by $\dizz$ the distance on
$\Cbu(\R;\ellz)$ defined by
\begin{equation*}
  \dizz(X,Y)=\sup_{t\in\R}\diz(X(t),Y(t)).
\end{equation*}

The next result shows that we can see $\shft$-almost
periodicity of a random process $X$ 
as ordinary almost periodicity of a function $Y$ with values in a
metric space. 

\begin{proposition}\label{prop:aptransl}
  Let $X :\,\R\rightarrow\ellz$ be a random process. 
  Assume that $X$ is continuous in probability. 
  Let $Y$ be the translation mapping
\begin{equation*}
  Y :\,
  \left\{ \begin{array}{lcl}
            \R&\rightarrow &\Cbu(\R;\ellz)\\
            t&\mapsto &Y(t)=\transl_tX
   \end{array}\right.      
\end{equation*}
  Then, for each $\epsilon>0$,
  $X$ and $Y$ have the same $\epsilon$-almost periods.
  Furthermore, 
  $X$ is $\shft$-almost periodic in the sense of Definition
  \ref{def:aprds} if, and only if,
$Y$ is almost periodic in the sense of Definition \ref{def:ap}.
\end{proposition}
\begin{proof}
  Let $\tau \in\R$. We have, for any $t\in\R$,
   \begin{align*}
    \dizz(Y(t+\tau),Y(t))
    =&\sup_{s\in\R}\diz(X(s+t+\tau,\shft_{-t-\tau}\point),X(s+t,\shft_{-t}\point))\\
    =&\sup_{s\in\R}\diz(X(s+\tau,\shft_{-\tau}\point),X(s,\point)).
   \end{align*}
   Since the last term is independent of $t$, this yields
   \begin{equation}\label{eq:XY}
     \sup_{t\in\R}\dizz(Y(t+\tau),Y(t))
    =\sup_{s\in\R}\diz(X(s+\tau,\shft_{-\tau}\point),X(s,\point)). 
  \end{equation}
  Thus, for every $\epsilon>0$,
  $\tau$ is a $\shft$-$\epsilon$-almost period of $X$
  if, and only if,
  it is an $\epsilon$-almost period of $Y$.

  If $X$ is $\shft$-almost periodic, then \eqref{eq:XY}
  and Part \eqref{item:unifcont}
 of Theorem \ref{theo:jointunifcont}
  show that
  $Y$ is continuous,
  thus $Y$ is almost periodic in
 the sense of Definition \ref{def:ap}.

 Conversely, if $Y$ is continuous, then \eqref{eq:XY}
 and the reasoning of
 the second part of the proof of Proposition \ref{cor:jointcont}
 show that
 $X$ satisfies
 Condition  \eqref{cond:ltap}-(i) 
 of Definition \ref{def:aprds}, thus it is $\shft$-almost periodic.
\end{proof}

\begin{remark}%
[Almost periodicity of the translate function
    does not depend on any uniform structure]
\label{rem:topological2}
    To define the state space of $Y$ in Proposition \ref{prop:aptransl}, 
we have used the topology of uniform convergence on $\Cb(\R;\ellz)$,
which is related to the distance $\diz$. So, it might seem that the
property of almost periodicity of $Y$ depends on the distance $\diz$.
But Remark \ref{rem:topological} combined with Proposition \ref{prop:aptransl}
show that this is not the case and that 
  any other distance than $\diz$ compatible with the topology of
  convergence in probability lets the almost periodicity property
  of $Y$ unchanged. 
\end{remark}

\section{Almost periodicity in distribution}\label{sec:distrib}

\subsection{Different notions of almost periodicity in distribution}
The following definitions are inspired by Tudor \cite{Tudor95ap_processes},
see also \cite{BMRF}.

In the sequel, we denote by
$\Cbk(\R;\esp)$ the space of continuous functions from $\R$ to
$\esp$
endowed with the compact-open topology, that is,
the topology of uniform convergence on
compact subsets (equivalently, on compact intervals) of $\R$.
It is well known that $\Cbk(\R;\esp)$ is Polish.

\begin{definition}\label{def:apdist}
  \begin{enumerate}[(a)]
  \item A random process $X:\,\R\rightarrow\ellz$
    is \emph{almost periodic in one-dimensional
      distribution}, or \emph{APOD}
    (respectively, \emph{periodic in one-dimensional
      distribution}, or \emph{POD})
if the mapping
\[\left\{ \begin{array}{lcl}
  \R&\rightarrow&\laws{\esp}\\
  t&\mapsto&\law{X(t,\point)}
\end{array}\right.\]
is almost periodic (resp.~periodic). 

\item The process $X$ is \emph{almost periodic in finite-dimensional
    distribution}, or \emph{APFD}
  (respectively, \emph{periodic in finite-dimensional
      distribution}, or \emph{PFD}),
if, for every finite sequence $(t_1,\dots,t_n)$, the mapping
\[\left\{ \begin{array}{lcl}
  \R&\rightarrow&\laws{\esp^n}\\
t&\mapsto&\law{\transl_{t}{X(t_1,\point)},\dots,\transl_{t}{X(t_n,\point)}}
\end{array}\right.\]
is almost periodic (resp.~$\tau$-periodic, for some $\tau>0$ which does
not depend on $(t_1,\dots,t_n)$). 

\item \label{def:APPD}
If the process $X$ has a version with continuous trajectories
(for simplicity, let us
denote by $X$ this version), we say
that $X$
is \emph{almost periodic in path distribution}, or
\emph{APPD}, 
if the mapping
\[\left\{ \begin{array}{lcl}
  \R&\rightarrow&\laws{\Cbk(\R;\esp)}\\
t&\mapsto&\law{\transl_tX}
\end{array}\right.\]
is almost periodic.
\end{enumerate}
\end{definition}

\begin{remark}\label{rem:apd}\rule{1em}{0em}
  \begin{enumerate}
  \item Clearly,
    $\mbox{APPD}\Rightarrow \mbox{APFD}\Rightarrow\mbox{APOD}$
  and $\mbox{PFD}\Rightarrow\mbox{POD}$.
The notion of periodicity in path distribution (PPD) is not
relevant since it is equivalent to PFD.
Indeed, assume that $X$ has a
continuous version and that
all finite distributions of $X$ are $\tau$-periodic.
Let $J=[t_0,t_0+T]$ be a fixed interval. 
Let $X_J$ be the random variable with values in $\Cbu(J,\esp)$
defined by
\[X_J(\omega)(t)= X(t,\omega)\quad (t\in J).\]
We define similarly the random variable $X_{J+\tau}$.
To prove that $X_J$ and $X_{J+\tau}$ have the same distribution,
we  embed isometrically  $\esp$ in
some separable Banach space $\espB$, see, 
e.g., \cite{heinonen} on such embeddings. 
For all integers $n\geq 1$ and $k=0,\dots,n$, set $t_k^n=t_0+kT/n$.
Let $X_J^n$ be the random variable with piecewise linear
values in $\Cbu(J,\espB)$
which coincides with $X_J$ at $t_0^n,\dots,t_n^n$.
We have
\begin{align*}
    &\lim_{n\rightarrow\infty}\law{X_J^n}=\law{X_J}
      \ \text{ and }\ 
      \lim_{n\rightarrow\infty}\law{X_{J+\tau}^n}=\law{X_{J+\tau}},
\end{align*}
since $X_J^n$ and $X_{J+\tau}^n$ converge respectively to $X_J$ and
$X_{J+\tau}$ a.e.~in $\Cbu(J,\espB)$.
Let $f :\,\Cbu(J,\espB)\rightarrow\R$ be a bounded continuous
function, and let $\epsilon>0$. For $n$ large enough, we have
\[\abs{\expect\CCO{f(X_{J}^n)-f(X_{J})}}\leq\epsilon
  \ \text{ and }\
\abs{\expect\CCO{f(X_{J+\tau}^n)-f(X_{J+\tau})}}\leq\epsilon.  
\]
But, by the periodicity assumption, we have 
$\expect\CCO{f(X_{J+\tau}^n)-f(X_{J}^n)}=0$, thus
\[\abs{\expect\CCO{f(X_{J+\tau})-f(X_{J})}}\leq2\epsilon.\]
The result follows since $\epsilon$ is arbitrary.

\item\label{item:semidistances} The topology of $\Cbk(\R;\esp)$
is defined in a natural way by a countable family of semidistances,
e.g., for $x,y\in\Cbk(\R;\esp)$,
\begin{equation*}
  \dist_k(x,y)=\sup_{t\in[-k,k]}\dist(x(t),y(t)),\quad(k\geq 1). 
\end{equation*}
It is also metrized by, e.g., the
  distance $\distCk$ defined by
  \begin{equation*}
    \distCk\CCO{x,y}=\sum_{k\geq 1}2^{-k}(\dist_k(x,y)\wedge 1).
  \end{equation*}
  Since $\distCk$ is bounded, 
  a distance on $\laws{\Cbk(\R;\esp)}$ associated with $\distCk$ is
  \begin{equation*}
   \wassCk= \inf_{\law{X}=\mu,\,\law{Y}=\nu}
    \expect{\distCk(X,Y)},
  \end{equation*}
  and a random process $X$ is APPD if, for each
  $\epsilon>0$, the set of $\epsilon$-periods of 
  $t\mapsto\law{\transl_tX}$ for $\wassCk$ is relatively dense.
  
  Another equivalent approach,
  thanks to Proposition \ref{prop:semidistances}, consists in
  defining the APPD property using semidistances on
  $\laws{\Cbk(\R;\esp)}$ associated with
  $\dist_k$, $k\geq 1$:
  $X$ is
  APPD if, and only if, for each integer $k\geq 1$,
  the map $t\mapsto\law{\transl_tX}$ is almost periodic for the
  semidistance
  \begin{equation*}
   \wassk{[-k,k]}(\mu,\nu)= \inf_{\law{X}=\mu,\,\law{Y}=\nu}
    \expect\CCO{\dist_k(X,Y)\wedge 1}.
  \end{equation*}

\item Definition \eqref{def:APPD} can be easily generalized to other
  spaces of trajectories as follows:
\begin{itemize}
\item[(c')]  If $X$ has a version (again denoted by $X$)
  whose trajectories lie in a metrizable topological space of
trajectories $\traj\subset\esp^{\R}$ which is stable by translations, that is, such
that $x\in\traj\Rightarrow x(t+\point)\in\traj$ for all $t\in\R$, 
we say
that $X$
is \emph{almost periodic in $\traj$-path distribution} (let us say,
\emph{$\traj$-APPD}) 
if the mapping
\[\left\{ \begin{array}{lcl}
  \R&\rightarrow&\laws{\traj}\\
t&\mapsto&\law{\transl_tX}
\end{array}\right.\]
is almost periodic.
\end{itemize}

\item Since $\shft$ is measure preserving, Definition
  \eqref{def:apdist} as well as Definition (c') above
  remain unchanged if we replace the operator
  $\transl_t$ by the simpler operator $\translsimple{t}$ defined by
  $\translsimple{t}X(s,\point)=X(t+s,\point)$ for all $s\in\R$. 

\end{enumerate}
\end{remark}

The following criterion from \cite[Theorem 2.3]{BMRF} is based on the
Arzel\`a-Ascoli theorem:
\begin{proposition}[\cite{BMRF}]\label{prop:apfd-appd}
  Let $X :\,\R\mapsto \ellz$ be a random process. 
  Assume that $X$ is APFD, and that $X$
  has a continuous modification.
Then $X$ is APPD if, and only if, it satisfies, 
for every compact interval $J$, 
  \begin{equation}\label{eq:tightC}
    \lim_{\delta\rightarrow 0}\sup_{t\in\R}
    \expect\left(
      \sup_{r,s\in J,\, \abs{r-s}<\delta}
      \dist\bigl(X(t+r,\point) , X(t+s,\point) \bigr)\wedge 1
    \right)=0.
  \end{equation}
\end{proposition}

\subsection{Comparison with $\shft$-almost periodicity}

For applications to stochastic differential equations, we will use
the following sufficient criterion for APPD property for
$\shft$-almost periodic processes.
For any $J\subset I$
be denote by
$\Cbu(J;\esp)$ the space of continuous functions from $J$ to
$\esp$
endowed with 
the topology of uniform convergence.
\begin{theorem}[$\shft$-almost periodicity vs almost
  periodicity in distribution]\label{theo:ap-vs-apdist}
  If $X$ is $\shft$-almost periodic (respectively $\shft$-periodic),
  it is APFD (respectively PFD).

  If furthermore $X$ has a continuous modification (that we denote by
  $X$ for simplicity),
  a sufficient condition for $X$ to be APPD is Condition
  (\ref{item:ap-APPD3}) below:
  \begin{enumerate}[(A)]
    \setcounter{enumi}{2}
  \item\label{item:ap-APPD3} For every compact interval $J$,
    the mapping $Z_J :\,\R\rightarrow \ellp{0}(\Cbu(J;\esp))$
    defined by 
\begin{equation}\label{eq:Z_J}
  Z_J(t)(\omega)(s)= \transl_tX(s,\omega)=X(t+s,\shft_{-t}\omega)
  \quad (t\in\R,\ \omega\in\Omega,\ s\in J)
\end{equation}
is almost periodic.
\end{enumerate}

\end{theorem}
\begin{remark}
  \begin{enumerate}
  \item By Proposition \ref{prop:aptransl}, 
    with the the notation of \eqref{item:ap-APPD3}, 
    $\shft$-almost periodicity amounts to almost periodicity of  
    $Z_{\{0\}}$. Thus
    Condition (\ref{item:ap-APPD3}) implies at the same time
    $\shft$-almost periodicity and APPD property.

\item  By Proposition \ref{prop:semidistances}
(see also Part \ref{item:semidistances} of Remark \ref{rem:apd}), 
\eqref{item:ap-APPD3} is equivalent to:
{\em 
\begin{itemize}
\item[{\it(\ref{item:ap-APPD3})'}]
  The mapping $Z :\,\R\rightarrow \ellp{0}(\Cbk(\R;\esp))$
    defined by 
\begin{equation*}
  Z(t)(\omega)(s)= \transl_tX(s,\omega)=X(t+s,\shft_{-t}\omega)
  \quad (t\in\R,\ \omega\in\Omega,\ s\in \R)
\end{equation*}
is almost periodic.
\end{itemize}
}
\end{enumerate}
\end{remark}
\begin{proof}[Proof of Theorem \ref{theo:ap-vs-apdist}]
Let $X$ be $\shft$-almost periodic.
  Let $(t_1,\dots,t_n)$ be a finite sequence in $\R$. Let us endow
  $\esp^n$ with the distance
  \[\dist_n\bigl((x_1,\dots,x_n),(y_1,\dots,y_n)\bigr))
  =\max_{1\leq i\leq n}\dist(x_i,y_i).\] 
  Let $\epsilon>0$, and let $\tau$ be an $\epsilon$-period of $X$. 
  For every $t\in\R$, we have
  \begin{multline*}
    \wass\Biggl(
    \law{\transl_{t_1}{X(t,\point)},\dots,\transl_{t_n}{X(t,\point)}},
    \law{\transl_{t_1+\tau}{X(t,\point)},\dots,\transl_{t_n+\tau}{X(t,\point)}}
    \Biggr)\\
    \leq 
    \expect\Biggl(\max_{1\leq i\leq n}
    \dist\Bigl(\transl_{t_i}{X(t,\point)},
                     \transl_{t_i+\tau}{X(t,\point)}\Bigr)\wedge 1
                     \Biggr)\\
    \leq 
    n \max_{1\leq i\leq n}\expect\Biggl(
    \dist\Bigl(\transl_{t_i}{X(t,\point)},
                     \transl_{t_i+\tau}{X(t,\point)}\Bigr)\wedge 1
                     \Biggr)
                     \leq n\epsilon,
  \end{multline*}
  which shows that $X$ is APFD.

If $X$ is $\shft$-$\tau$-periodic, the same reasoning with
$\epsilon=0$ yields the PFD property.

Assume now that $X$ has a continuous modification,
that we also denote by $X$ for
simplicity. 
Assume \eqref{item:ap-APPD3}, and let $J\subset \R$ be a compact
interval.
Let $\epsilon>0$.
Let $l$ such that each interval of length $l$ contains an $\epsilon/3$-almost
period of $Z_J$. We can choose $l$ large enough that
$J\subset[-l/2,l/2]$. Let $I=[-l,l]$. 
By uniform continuity on $I$ of the trajectories of $X$,
we can find an $\tribu$-measurable random variable
$\eta(\point)>0$ such that, for every $\omega\in\Omega$,
\begin{equation*}
  \bigl(r,s\in I\text{ and }\abs{r-s}\leq\eta(\omega)\bigr)
  \Rightarrow
  \dist\bigl(X(r,\omega),X(s,\omega)\bigr)\leq\frac{\epsilon}{6}.
\end{equation*}
By tightness of $\eta(\point)$, we can find a number $\delta>0$ such that
\begin{equation*}
  \prob\CCO{\eta(\point)\geq \delta}\geq1-\frac{\epsilon}{6}.
\end{equation*}
Let $A=\{\eta(\point)\geq \delta\}$ and $A^c=\Omega\setminus
A$.
We have thus
\begin{multline*}
  \expect\CCO{\sup_{r,s\in I,\, \abs{r-s}\leq\delta}
    \dist\bigl(X(r,\point),X(s,\point)\bigr)\wedge 1}\\
  \begin{aligned}
  \leq&
  \expect\CCO{\sup_{r,s\in I,\, \abs{r-s}\leq\delta}
    \dist\bigl(X(r,\point),X(s,\point)\bigr)\un{A}\wedge 1}\\
 & +\expect\CCO{\sup_{r,s\in I,\, \abs{r-s}\leq\delta}
    \dist\bigl(X(r,\point),X(s,\point)\bigr)\un{A^c}\wedge
    1}\\
  \leq& \frac{\epsilon}{6}+\prob\CCO{A^c}\leq \frac{\epsilon}{3}.
  \end{aligned}
\end{multline*}
Let $t\in\R$, and let $\tau\in [-t-l/2,-t+l/2]$
be an $\epsilon/3$-almost period of
$Z_J$, so that $J+t\subset I$. We have
\begin{equation*}
  \expect\CCO{\sup_{s\in J}
    \dist\CCO{\transl_{t}(X(s,\point)),\transl_{t+\tau}(X(s,\point))}}
  \leq\frac{\epsilon}{3}.
\end{equation*}
On the other hand, we can find measurable random variables
$r(\point)$ and $s(\point)$ with values in $J$ such that
$\abs{r(\omega)-s(\omega)}\leq \delta$ for each $\omega\in\Omega$
and
\begin{multline*}
  \expect\left(
      \sup_{r,s\in J,\, \abs{r-s}<\delta}
      \dist\bigl(X(t+r,\point) , X(t+s,\point) \bigr)\wedge 1
    \right)\\
    =
    \expect\left(
      \dist\bigl(X(t+r(\point),\point) , X(t+s(\point),\point) \bigr)\wedge 1
    \right).
\end{multline*}
Using that $\prob$ is $\shft$-invariant, we get thus
\begin{multline*}
  \expect\CCO{\sup_{r,s\in J,\, \abs{r-s}\leq\delta}
    \dist\bigl(X(t+r,\point),X(t+s,\point)\bigr)\wedge 1}\\
  \begin{aligned}
  =&\expect\CCO{
    \dist\bigl(X(t+r(\point),\shft_{-t}\point),
                   X(t+s(\point),\shft_{-t}\point)\bigr)\wedge 1}\\
  \leq& \expect\CCO{
    \dist\bigl(X(t+r(\point),\shft_{-t}\point),
    X(t+\tau+r(\point),\shft_{-t-\tau}\point)\bigr)\wedge
    1}\\
  &+\expect\CCO{
    \dist\bigl(X(t+\tau+r(\point),\shft_{-t-\tau}\point),
    X(t+\tau+s(\point),\shft_{-t-\tau}\point)\bigr)\wedge 1}\\
  &+\expect\CCO{
    \dist\bigl(X(t+\tau+s(\point),\shft_{-t-\tau}\point),
    X(t+s(\point),\shft_{-t}\point)\bigr)\wedge 1}\\
  \leq& \expect\CCO{
    \dist\bigl(X(t+r(\point),\shft_{-t}\point),
    X(t+\tau+r(\point),\shft_{-t-\tau}\point)\bigr)\wedge
    1}\\
  &+\expect\CCO{\sup_{r,s\in I,\, \abs{r-s}\leq\delta}
    \dist\bigl(X(r,\point),
    X(s,\point)\bigr)\wedge 1}\\
  &+\expect\CCO{
    \dist\bigl(X(t+\tau+s(\point),\shft_{-t-\tau}\point),
    X(t+s(\point),\shft_{-t}\point)\bigr)\wedge
    1}\\
  &\leq\epsilon.
  \end{aligned}
\end{multline*}
Since this estimation is independent of $t$, this proves 
Condition \eqref{eq:tightC} of Proposition \ref{prop:apfd-appd}.
\end{proof}

The following counterexample is inspired from Ursell \cite{ursell}. It shows
that $\shft$-almost periodicity does not imply the APPD property, in
particular this property is strictly stronger than the APFD property.

\begin{counterexample}
  Let $\Omega=[0,1]$, endowed with Lebesgue measure. let
  $(\epsilon_n)_{n\geq 1}$ be a sequence in $]0,1]$ such that
  $\sum_n\epsilon_n<\infty$. For each positive integer $n$ and each
  integer $k$ (positive or nonpositive), let $x_{n,k}=(2k+1)n$, and
  define
 \begin{gather*}
   f_n(t)=\sum_{-\infty<k<\infty}\CCO{
     \frac{1}{\epsilon_n}-\abs{\frac{t}{\epsilon_n}-x_{n,k}}}
     \un{\{\abs{t-x_{n,k}}\leq \epsilon_n\}},\\
    f(t)=\sum_{n\geq 1}f_n(t).
  \end{gather*}
  Each $f_n$ is periodic and continuous, and $f$ is continuous,
  but $f$ is not uniformly
 continuous, nor bounded, thus it is not almost periodic in Bohr's sense.
 However, $f$ is almost periodic in Stepanov's sense, that is, $f$ is
 locally integrable and the
 mapping
 \[
   \left\{
     \begin{array}{lcl}
       \R&\mapsto&\ellp{1}([0,1])\\
       t&\mapsto&f(t+\point)
     \end{array}
     \right.
\]
is almost periodic.

Set $X(t,\omega)=f(t+\omega)$ for $t\in\R$ and $\omega\in\Omega$. Then
$X$ is $\shft$-almost periodic for the shift transformation
\[\shft_t(\omega)=t+\omega\quad \mod{1}.\]
Let $J=[0,2]$. Set, for $n\geq 1$, $\omega\in\Omega$ and $\delta\in]0,1]$, 
\[t_n=n(2n+1)-1,\ r_n(\omega)=1-\omega,\
  s_n(\omega)=1-\delta-\omega. \]
We have
\begin{align*}
  \sup_{t\in\R}
    \expect&\left(
      \sup_{r,s\in J,\, \abs{r-s}<\delta}
      \dist\bigl(X(t+r,\point) , X(t+s,\point) \bigr)\wedge 1
             \right)\\
  =&\sup_{t\in\R}
    \int_0^1\left(
      \sup_{r,s\in J,\, \abs{r-s}<\delta}
      \abs{f(t+r+\omega) - f(t+s+\omega)}\wedge 1
             \right)d\omega\\
  \geq&\sup_{n\geq 1}\int_0^1\left(
      \abs{f(t_n+r_n(\omega)+\omega) - f(t_n+s_n(\omega)+\omega)}\wedge 1
        \right)d\omega\\
  \geq&\sup_{n\geq 1}\frac{1}{\epsilon_n}\wedge 1=1,
\end{align*}
which shows by Proposition \ref{prop:apfd-appd} that $X$ is not APPD. 
\end{counterexample}

\section{Application to stochastic differential equations}\label{sec:SDE}

We apply here the results of the previous sections. 
The novelty is that 
we can reduce the proof of almost periodicity to a fixed point
problem, unlike in usual proofs of almost periodicity in distribution.
Recall that the ``naive`` almost periodicity in $p$-mean 
(that corresponds to $\shft_t=\Id_\Omega$, or to $\linshft(t)=0$
in Remark \ref{rem:def:rds}-\eqref{rem:choice_of_l})
does not apply to stochastic differential equations 
\cite{almost-automorphic,counterexamples}. 

In this section, 
$\U$ and $\espH$ are separable Hilbert spaces, 
$\Omega=\Cb(\R;\U)$ is endowed
with the compact-open topology (the topology of uniform convergence on
compact subsets of $\R$), $\tribu$ is the Borel $\sigma$-algebra of
$\Omega$, and $\prob$ is the Wiener measure on $\Omega$
with trace class\footnotemark covariance operator $Q$, that is,
the process $W$ with values in $\U$ defined by 
\begin{equation*}
  W(t,\omega)=\omega(t),\quad \omega\in\Omega,\ t\in\R,
\end{equation*}
is a Brownian motion with covariance operator $Q$.\footnotetext{It is
  possible to consider a more general,
  nonnecessarily nuclear, nonegative symmetric
   operator $Q$.
   If $\trace Q=+\infty$, $W$ is only a cylindrical Brownian motion on
  $\U$. However, it is possible to embed $\U$ in a ``larger'' Hilbert space
$\U_1$ such that $W$ has trace class covariance operator on
$\U_1$. In that case, we have to take $\Omega=\Cb(\R;\U_1)$ and
$\prob$ is the
distribution of $W$ on $\U_1$, see \cite[Chapter 4]{dapratozabczyk}.} 
Our group $\shft$ of measure preserving transformations 
on $(\Omega,\tribu,\prob)$ is the {\em Wiener shift}, defined by
\begin{equation}\label{eq:Wienershift}
  \shft_\tau(\omega)(t)=\omega(t+\tau)-\omega(\tau)
                            =W(t+\tau,\omega)-W(\tau,\omega)
\end{equation}
for all $\tau,t\in\R$ and $\omega\in\Omega$. This yields
\begin{equation}\label{eq:chvar}
  W(t+\tau,\shft_{-\tau}\omega)=(\shft_{-\tau}\omega)(t+\tau)
  =\omega(t+\tau-\tau)-\omega(-\tau)
  =W(t,\omega)-W(-\tau,\omega),
\end{equation}
 so that
the translation operator of Definition \ref{def:translate}
leaves invariant the increments of $W$: Indeed, 
we have, for all $t,s,\tau\in\R$,
\begin{multline*}
  \transl_\tau (W(t+s,\omega)-W(t,\omega))\\
\begin{aligned}
  =&\Bigl(W(t+s,\omega)-W(-\tau,\omega)\Bigr)
    -\Bigl(W(t,\omega)-W(-\tau,\omega) \Bigr)\\
  =&W(t+s,\omega)-W(t,\omega).
     \end{aligned}
\end{multline*}
We refer to Da Prato and Zabcyk's treatise \cite{dapratozabczyk}
for stochastic integration and stochastic differential equations
in Hilbert spaces. 
For the needs of stochastic integration with respect to $W$, we endow
$(\Omega,\tribu,\prob)$ with the augmented natural filtration
$(\tribu_t)$ of $W$. 
We denote by $\U_0$ the Hilbert space $Q^{1/2}\U$ with norm
$\norm{u}_{\U_0}=\norm{Q^{-1/2}}_{\U}$. 
Let $\HSz$ be the space of Hilbert-Schmidt operators from $\U_0$ to
$\espH$. 
We consider the semilinear stochastic differential equation, for
$t\in\R$,  
\begin{equation}\label{eq:EDSfond}
  dX(t,\point)=\bigl(A X(t,\point)+F(t,X(t,\point))\bigr)\,dt
                            +G(t,X(t,\point))\,dW(t,\point)
\end{equation}
where the unknown process $X$ takes its values in  $\espH$,
$A :\,\domain(A)\subset\espH\rightarrow\espH$ is a 
linear operator which may be unbounded, 
and $F: \R\times\espH \rightarrow\espH$, 
and $G: \R\times\espH\rightarrow \HSz$ 
are continuous functions. 
We assume that
\begin{enumerate}[(H1)]
\item\label{hyp:C0} $A$ is the
infinitesimal generator of  
a $C_0$-semigroup $(S(t))_{t\geq0}$
satisfying, for some constant numbers $\ddelta>0$, and $M>0$,
and for all $t\geq 0$,
\begin{equation}\label{eq:Scontractant}
  \|S(t)\|_{\Linear(\espH)}\leq M e^{-\ddelta t},
  \end{equation}
where $\Linear(\espH)$ is the space of continuous linear mappings from
$\espH$ to itsef,

\item\label{hyp:FG}  $F$ and $G$ satisfy
the usual Lipschitz and growth conditions, that is,
for some constant numbers $\ctgrowth,\ctlip>0$ and for all $x,y\in\espH$,
\begin{gather}
  \norm{F(t,x)-F(t,y)}+\norm{G(t,x)-G(t,y)}\leq \ctlip\norm{x-y},
                                                     \label{eq:lip}\\
  \norm{F(t,x)}+\norm{G(t,x)}\leq \ctgrowth(1+\norm{x}),\label{eq:growth}
\end{gather}

\item\label{hyp:pp}  $F$ and $G$ are almost periodic
  uniformly with respect to compact
  subsets of $\espH$ (see Definition \ref{def:ap}). 

\end{enumerate}
A \emph{mild solution} to \eqref{eq:EDSfond} is a random process $X$
such that,
for all $t,s\in\R$ such that $s\leq t$, we have
$\expect\int_s^t\norm{X(u,\point)}^2du<\infty$, and
\begin{multline}\label{eq:mildsol}
  X(t,\point)=S(t-s)X(s,\point)+\int_s^tS(t-u)F(u,X(u,\point))\,du\\
  +\int_s^tS(t-u)G(u,X(u,\point))\,dW(u,\point).
\end{multline}
By \cite[Theorem 7.2]{dapratozabczyk},
for any $t_1\in\R$  and any $\xi\in\ellp{2}(\Omega,\tribu_t;\espH)$,
there exists a mild solution to Equation
\eqref{eq:EDSfond} which is defined on $[t_1,\infty[$ and starts from
$\xi$ at $t_1$, and 
each mild solution has a continuous version. 
Furthermore, Equation \eqref{eq:EDSfond} defines an ``almost flow''
(see the discussion in \cite[Section 9.1.2]{dapratozabczyk}), that is,
a mapping
\begin{equation*}
  \flow :\,\left\{\begin{array}{lcl}
    \R\times\R\times\Omega\times\espH&\rightarrow&\espH\\                
    (t_2,t_1,\omega,x)&\mapsto&\flow(t_2,t_1,\omega)x                
  \end{array}\right.  
\end{equation*}
defined for $t_2\geq t_1$, and satisfying
$\flow(t_1,t_1,\omega)=\Id_{\espH}$ and 
\begin{equation}\label{eq:almostflow}
  \flow(t_3,t_1,\omega)=\flow(t_3,t_2,\omega)\circ \flow(t_2,t_1,\omega)
\end{equation}
for $t_1\leq t_2\leq t_3$ and $\prob$-almost all $\omega\in\Omega$.
Denote $\tau=t_1$, $s=t_2-t_1$, $r=t_3-t_2$, and
\begin{equation*}
  \rds(s,\tau,\shft_{t}\omega)=\flow(\tau+s,\tau,\omega),
\end{equation*}
where $\shft$ is defined as in \eqref{eq:Wienershift},
then 
Equation \eqref{eq:almostflow} becomes the crude cocycle relation
\eqref{eq:crude}. Any mild solution which is defined on $\R$ is an
orbit of $\rds$.

We seek solutions that are defined on $\R$  and
$\shft$-almost periodic in quadratic mean. 
A solution $X$ of this type is 
uniformly bounded in $\ellp{2}(\Omega;\espH)$. 
By \eqref{eq:Scontractant}, it satisfies
\begin{multline}\label{eq:mild}
  X(t,\point)=\int_{-\infty}^tS(t-u)F(u,X(u,\point))\,du
  +\int_{-\infty}^tS(t-u)G(u,X(u,\point))\,dW(u,\point).
\end{multline}
In the sequel, we denote
\begin{itemize}

\item $\CUB$ the space of continuous uniformly bounded functions from $\R$ to
  $\ellp{2}(\Omega;\espH)$, endowed with the norm
  $\norm{X}=\sup_{t\in\R}\norm{X(t)}_{\ellp{2}(\Omega;\espH)}$, 

\item $\APB\subset\CUB$ the space of $\shft$-almost periodic in square mean
$(\tribu_t)$-predictable random processes,

\item $\PB(\tau)\subset \APB$ the space of  $\shft$-$\tau$-periodic and
  square integrable $(\tribu_t)$-predictable random processes. 
\end{itemize}
Any constant (in $t$ and $\omega$)
random process belongs to each of these spaces,
thus none of them is empty. 
We define an operator $\SDE :\, \APB\rightarrow\CUB$ by 
\begin{equation*}
   \SDE X(t)=\int_{-\infty}^t S(t-s)F\bigl(s, X(s)\bigl)ds + 
\int_{-\infty}^t S(t-s)G\bigl(s, X(s)\bigl)dW(s)
\end{equation*}
for all $t\in\R$.
By \cite[Proposition 7.3]{dapratozabczyk}, for any $X\in\APB$,
the random process $\SDE X$ has a continuous
modification. Furthermore, 
for any $u\in\R$,
 the family $(X(u+s,\shft_{-s}\point))_{s\in\R}$ is
uniformly square integrable, since it is relatively compact in
$\ellp{2}(\Omega,\prob;\espH)$. Since $\shft$ is measure preserving,
this entails that the family $(X(s,\point))_{s\in\R}$  is
uniformly square integrable. Using Vitali's theorem, we deduce that  
$\SDE X$ is continuous in square mean.

Combining $\SDE$ with the translation operator of Definition
\ref{def:translate}, we get
\begin{proposition}\label{prop:apstable}
  The operator $\SDE$ maps $\APB$ into itself. 
\end{proposition}
\begin{proof}
  Let $X\in\APB$ be a
  $\shft$-almost periodic $(\tribu_t)$-predictable random process.
Let $\epsilon_0>0$, and let us show that the set of
$\shft$-$\epsilon_0$-almost periods of $\SDE X$ is relatively dense. 
  
As we already noticed, 
the family $(X(s,\point))_{s\in\R}$  is
uniformly square integrable.
From the growth condition \eqref{eq:growth}, we deduce that
the families $(F(r,X(s,\point)))_{r,s\in\R}$ and
$(G(r,X(s,\point)))_{r,s\in\R}$ are also uniformly square integrable.
Let $\alpha>0$.
There exists $\eta>0$, with $\eta<\min(\alpha,1)$,
such that, for any $A\in\tribu$, and for
all $u,s\in\R$,
\begin{equation}\label{eq:Xui}
  \prob(A)<\eta \Rightarrow\left\{ \begin{array}{l}
  \expect\CCO{ \norm{X(u+s,\shft_{-s}\point)}^2
                                     \un{A} }<{\alpha},\\
  \expect\CCO{ \norm{F(r,X(u+s,\shft_{-s}\point))}^2
                                     \un{A} }<{\alpha},\\
   \expect\CCO{ \norm{G(r,X(u+s,\shft_{-s}\point))}^2
                                     \un{A} }<{\alpha}.
  \end{array} \right.
\end{equation}
Furthermore, by Proposition \ref{prop:tightness},
the family $(X(s,\point))_{s\in\R}$ is uniformly
tight. 
We can find thus a compact subset $K_\alpha$ of $\espH$ such that, for
all $u,s\in\R$,
\begin{equation}\label{eq:Xtight}
  \prob\accol{X(u+s,\shft_{-s}\point)\in K_\alpha}\geq 1-\eta.
\end{equation}
On the other hand, 
we deduce from Corollary \ref{cor:ap-product} and Proposition
  \ref{prop:aptransl},
  that the mapping
  \[
\left\{ \begin{array}{lcl}
          \R&\rightarrow& \espH\times \HSz\times\ellzp{2}\\
          t&\mapsto&\Bigl(F(t,x),G(t,x),X(t,\point)\Bigr)
   \end{array}\right.  
  \]
  is $\shft$-almost periodic 
  uniformly with respect to $x$ in compact subsets
of $\espH$. 
Let $\epsilon>0$, and let $\aperiods{\epsilon}$ be the relatively dense
set of common $\epsilon$-almost periods of
$X$, $F(\point,x)$ and $G(\point,x)$ for all $x\in K_\alpha$. 
We are going to show that, for an appropriate choice of $\epsilon$ and
$\alpha$, the set 
$\aperiods{\epsilon}$ is contained in the
set of $\epsilon_0$-almost periods in square mean of $\SDE X$. 

Let $\tau\in\aperiods{\epsilon}$. Assume first that $\tau>0$. 
We have, using \eqref{eq:chvar},
\begin{equation}\label{eq:changeofvar}
\begin{aligned}
  \backshift_\tau \SDE X(t,\point)
  =&\int_{-\infty}^{t+\tau}
     S(t-s)F(s,X(s,\shft_{-\tau}\point))\,ds\\
    & +\int_{-\infty}^{t+\tau}
     S(t-s)G(s,X(s,\shft_{-\tau}\point))\,d W(s,\shft_{-\tau}\point)\\
  =&\int_{-\infty}^{t}
     S(t-u)F(u+\tau,X(u+\tau,\shft_{-\tau}\point))\,du\\
    & +\int_{-\infty}^{t}
     S(t-u)G(u+\tau,X(u+\tau,\shft_{-\tau}\point))\,d
      W(u+\tau,\shft_{-\tau}\point)\\
   =&\int_{-\infty}^{t}
     S(t-u)F(u+\tau,X(u+\tau,\shft_{-\tau}\point))\,du\\
    & +\int_{-\infty}^{t}
     S(t-u)G(u+\tau,X(u+\tau,\shft_{-\tau}\point))\,d
      W(u,\point).
\end{aligned}
\end{equation}    
We deduce
\begin{align*}
  \expect&{\norm{\backshift_{\tau} \SDE X(t,\point)-\SDE X(t,\point)}^2}\\
    =&\expect\Bigl\Vert
 \int_{-\infty}^{t}
        S(t-u)\Bigl(F(u+\tau,X(u+\tau,\shft_{-\tau}\point))
               -F(u,X(u,\point))\Bigr)
               \,du\\
       &+\int_{-\infty}^{t}
        S(t-u)\Bigl(G(u+\tau,X(u+\tau,\shft_{-\tau}\point))
               -G(u,X(u,\point))\Bigr)
        \,dW(u,\point)
    \Bigr\Vert^2\\
  \leq&4\expect\Bigl\Vert
        \int_{-\infty}^{t}
        S(t-u)\Bigl(F(u+\tau,X(u+\tau,\shft_{-\tau}\point))
               -F(u,X(u+\tau,\shft_{-\tau}\point))\Bigr)
        \,ds
        \Bigr\Vert^2\\
        &+4\expect\Bigl\Vert
        \int_{-\infty}^{t}
        S(t-u)\Bigl(F(u,X(u+\tau,\shft_{-\tau}\point))
               -F(u,X(u,\point))\Bigr)
        \,du
        \Bigr\Vert^2\\
        &\begin{aligned}
        +4\expect\Bigl\Vert
               \int_{-\infty}^{t}
        S(t-u)\Bigl(&G(u+\tau,X(u+\tau,\shft_{-\tau}\point))\\
               &-G(u,X(u+\tau,\shft_{-\tau}\point))\Bigr)
        \,dW(u,\point)
        \Bigr\Vert^2
        \end{aligned}\\
        &+4\expect\Bigl\Vert
        \int_{-\infty}^{t}
        S(t-u)\Bigl(G(u,X(u+\tau,\shft_{-\tau}\point))
               -G(u,X(u,\point))\Bigr)
        \,dW(u,\point)
        \Bigr\Vert^2\\
        =:&4I_1+4I_2+4I_3+4I_4.
 \end{align*}
 Let us denote, for $u\in\R$,
 \begin{align*}
   \Ja(u,\point)&=F(u+\tau,X(u+\tau,\shft_{-\tau}\point))
           -F(u,X(u+\tau,\shft_{-\tau}\point)),\\
   \Jb(u,\point)&=G(u+\tau,X(u+\tau,\shft_{-\tau}\point))
           -G(u,X(u+\tau,\shft_{-\tau}\point)),\\
   A_\alpha(u)&=\accol{\omega\in\Omega\tq
                X(u+\tau,\shft_{-\tau}\omega)\in K_\alpha
                 }. 
 \end{align*}
 Since $\tau$ is an $\epsilon$-period of $F(\point,x)$ and
 $G(\point,x)$, uniformly with respect to $x\in K_\alpha$, we have
\begin{equation} \label{eq:HKalpha}
  { \norm{\Ja(u,\point)}
        \un{A_\alpha(u)} }\leq\epsilon \text{ and }
   { \norm{\Jb(u,\point)}
        \un{A_\alpha(u)} }\leq\epsilon.
\end{equation}
 On the other hand, from \eqref{eq:Xtight}, we get
\begin{equation*}
  \prob\accol{A_\alpha(u)}\geq 1-\eta.
\end{equation*}
We deduce, by \eqref{eq:Xui},
\begin{equation}\label{eq:Hui-ab}
  \begin{aligned}
  \expect\CCO{ \norm{\Ja(u)}^2
    \un{\Omega\setminus A_\alpha(u)}}
\leq& 2\expect\CCO{ \norm{F(u+\tau,X(u+\tau,\shft_{-\tau}\point))}^2
    \un{\Omega\setminus A_\alpha(u)}}\\
    &+2\expect\CCO{ \norm{F(u,X(u+\tau,\shft_{-\tau}\point))}^2
      \un{\Omega\setminus A_\alpha(u)}}\\
 \leq&4\alpha,\\
      \text{ and }
   \expect\CCO{ \norm{\Jb(u)}^2
     \un{\Omega\setminus A_\alpha(u)} }\leq&4\alpha.
   \end{aligned}
   \end{equation}
 We deduce from \eqref{eq:HKalpha} and \eqref{eq:Hui-ab}, using
 Jensen's inequality,
  \begin{align}\label{eq:I1}
I_1= &\expect \Bigl\Vert
        \int_{-\infty}^{t}
        S(t-u) \Ja(u)
        \,du
       \Bigr\Vert^2\\
    =&\ddelta^2\expect \Bigl\Vert
             \int_{-\infty}^{t}
        \Bigl(S(t-u) \Ja(u)e^{\ddelta(t-u)}\Bigr)\frac{1}{\ddelta}e^{-\ddelta(t-u)}
        \,du
       \Bigr\Vert^2\\
 \leq &\ddelta^2\expect 
             \int_{-\infty}^{t}\Bigl\Vert
        S(t-u) \Ja(u)e^{\ddelta(t-u)}\Bigr\Vert^2\frac{1}{\ddelta}e^{-\ddelta(t-u)}
        \,du
    \\
    \leq &\ddelta^2M^2  \int_{-\infty}^{t}
           \Bigl(e^{-2\ddelta(t-u)}\lVert \Ja(u)\rVert^2e^{2\ddelta(t-u)}\Bigr)
           \frac{1}{\ddelta}e^{-\ddelta(t-u)}
        \,du
    \\
    =&\ddelta M^2  \int_{-\infty}^{t}
       e^{-\ddelta(t-u)}\lVert \Ja(u)\rVert^2      \,du
    \\
   \leq &{M^2}\CCO{\epsilon^2+4\alpha}.
 \end{align}
 Furthermore, using It\^o's isometry, we get
 \begin{align*}
   I_3=&\expect\Bigl\rVert
               \int_{-\infty}^{t}S(t-u)\Jb(u)\,dW(u,\point)\Bigr\rVert^2
        \leq M^2 \expect\int_{-\infty}^{t}e^{-2\ddelta(t-u)}
   \lVert \Jb(u)\rVert^2du\\
   \leq&\frac{M^2}{2\ddelta}\CCO{\epsilon^2+4\alpha}.
 \end{align*} 
To estimate $I_2$ and $I_4$, we use the Lipschitz condition
\eqref{eq:lip}: 
\begin{align*}
  I_2
  \leq &\ddelta M^2\int_{-\infty}^{t} e^{-\ddelta(t-u)}
         \expect\Bigl\lVert
   F(u,X(u+\tau,\shft_{-\tau}\point))
         -F(u,X(u,\point))\Bigr\rVert^2\,du\,\\
  \leq &\ddelta M^2\left(\int_{-\infty}^{t} e^{-\ddelta(t-u)}\,du\right)\,
 \sup_{u\in\R}\expect\Bigl\lVert
   F(u,X(u+\tau,\shft_{-\tau}\point))
         -F(u,X(u,\point))\Bigr\rVert^2\\
  =&{M^2}
   \sup_{u\in\R}\expect\Bigl\lVert
   F(u,X(u+\tau,\shft_{-\tau}\point))
         -F(u,X(u,\point))\Bigr\rVert^2\\  
\leq &  {M^2\ctlip^2\epsilon^2},\\
   I_4\leq &  M^2\expect
        \int_{-\infty}^{t}
             e^{-2\ddelta(t-u)}\Bigl\lVert
             G(u,X(u+\tau,\shft_{-\tau}\point))
               -G(u,X(u,\point))\Bigr\rVert^2
             \,du\\
  \leq & \frac{M^2\ctlip^2\epsilon^2}{2\ddelta}.
 \end{align*}
Gathering the estimations for $I_1$--$I_4$, we get 
\begin{equation}\label{eq:I1-I4}
  \expect{\norm{\backshift_{\tau} \SDE X(t,\point)-\SDE
      X(t,\point)}^2}
  \leq4M^2\CCO{{1}+\frac{1}{2\ddelta}}
          \CCO{\epsilon^2(1+\ctlip^2)+4\alpha}.
\end{equation}
If $\tau<0$, we make the change of variables
  $\tau'=-\tau$, $t'=t+\tau$.
Then
\begin{align*}
  \expect{\norm{\backshift_{\tau} \SDE X(t,\point)-\SDE
      X(t,\point)}^2}
  =&\expect{\norm{(\SDE X)(t+\tau,\shft_{-\tau}\point)-\SDE
     X(t,\point)}^2}\\
  =&\expect{\norm{(\SDE X)(t',\shft_{\tau'}\point)-\SDE
     X(t'+\tau',\shft_{-\tau'}\shft_{\tau'}\point)}^2}\\
  =&\expect{\norm{(\SDE X)(t',\point)-\SDE
     X(t'+\tau',\shft_{-\tau'}\point)}^2}\\
  =& \expect{\norm{\SDE X(t',\point)
     -\backshift_{\tau'} \SDE X(t',\point)}^2},
\end{align*}
so that the preceding calculation leads again to \eqref{eq:I1-I4}. 
Finally, choosing $\epsilon$ and $\alpha$ such that
\begin{equation*}
4M^2\CCO{{1}+\frac{1}{2\ddelta}}
          \CCO{\epsilon^2(1+\ctlip^2)+4\alpha}
   \leq\epsilon_0^2,     
        \end{equation*}
we have that $\tau$ is an $\epsilon_0$-almost period in square mean of
$\SDE X$.
Thus the set of  $\shft$-$\epsilon_0$-almost periods in square mean
of $\SDE X$ is contained in $\aperiods{\epsilon}$, thus it is 
relatively dense.

Finally,
there remains to prove the continuity condition \eqref{cond:ltap}-(i)
of definition \ref{def:aprds}.
By Proposition \ref{cor:jointcont}, we only need to
prove continuity at any $s_0\in\R$ of
$s\mapsto \SDE X(s,\shft_{-s}\point)=\transl_s \SDE X(0,\point)$.
The calculation follows similar lines as above.
Let $\alpha>0$, and let $\eta>0$ and $K_\alpha$
as in \eqref{eq:Xui} and \eqref{eq:Xtight}.
Using 
the calculation in  and
Set
\[A_\alpha(u)=\accol{\omega\in\Omega\tq
                X(u+s_0,\shft_{-s_0}\omega)\in K_\alpha
              }.\]
By the almost periodicity hypotheses (H\ref{hyp:pp}), $F$ and $G$ are
uniformly continuous on $\R\times K_\alpha$
(see, e.g., the proof of \cite[Theorem 2.3]{corduneanu}).
We can thus choose
$\eta$ such that
\begin{equation*}
  \abs{s-s_0}<\eta\Rightarrow
  \left\{ \begin{array}{l}
           \sup_{x\in K_\alpha}\norm{F(s,x)-F(s_0,x)}^2<\alpha\\
           \sup_{x\in
            K_\alpha}\norm{G(s,x)-G(s_0,x)}^2<\alpha\\
            \sup_{u\in\R}\expect\norm{X(u+s,\shft_{-s}\point))
                    -X(u+s_0,\shft_{-s_0}\point)}^2<\alpha.
         \end{array} \right.
     \end{equation*}
For $s_0>0$, using \eqref{eq:changeofvar} and similar calculations as
in \eqref{eq:I1}, we have that,
for any $s>0$,
\begin{align*}
 \expect&\norm{\SDE X(s,\shft_{-s}\point)-\SDE X(s_0,\shft_{-s_0}\point)}^2\\
  \leq &4M^2\Biggl(
         \ddelta\int_{-\infty}^{0}e^{\ddelta u}\expect \Bigl\lVert
         F(u+s,X(u+s,\shft_{-s}\point))
                      {-F(u+s,X(u+s_0,\shft_{-s_0}\point))\Bigr\rVert^2\,du}\\
  &+\ddelta\int_{-\infty}^{0}e^{\ddelta u}\expect \Bigl\lVert
         F(u+s,X(u+s_0,\shft_{-s_0}\point))
                      {-F(u+s_0,X(u+s_0,\shft_{-s_0}\point))\Bigr\rVert^2\,du}\\
&+\int_{-\infty}^{0}e^{2\ddelta u}\expect \Bigl\lVert
         G(u+s,X(u+s,\shft_{-s}\point))
                      {-G(u+s,X(u+s_0,\shft_{-s_0}\point))\Bigr\rVert^2\,du}\\
  &+\int_{-\infty}^{0}e^{2\ddelta u}\expect \Bigl\lVert
         G(u+s,X(u+s_0,\shft_{-s_0}\point))\\
                   &\pushright
                      {-G(u+s_0,X(u+s_0,\shft_{-s_0}\point))\Bigr\rVert^2\,du
                      \Biggr)}\\
  \leq &{4M^2\ctlip^2}{\ddelta}
         \int_{-\infty}^{0}e^{\ddelta u}\expect \Bigl\lVert
         X(u+s,\shft_{-s}\point))-X(u+s_0,\shft_{-s_0}\point))
         \Bigr\rVert^2\,du\\
  &+{4M^2}{\ddelta}\int_{-\infty}^{0}e^{\ddelta u}\expect \Bigl\lVert
         F(u+s,X(u+s_0,\shft_{-s_0}\point))\\
                   &\pushright
                     {-F(u+s_0,X(u+s_0,\shft_{-s_0}\point))\Bigr\rVert^2
                         \un{A_\alpha(u)}\,du}\\
 & +{4M^2}{\ddelta}\int_{-\infty}^{0}e^{\ddelta u}\expect \Bigl\lVert
         F(u+s,X(u+s_0,\shft_{-s_0}\point))\\
                   &\pushright
                     {-F(u+s_0,X(u+s_0,\shft_{-s_0}\point))\Bigr\rVert^2
                         \un{\Omega\setminus A_\alpha(u)}\,du}\\
 &+4M^2\ctlip^2\int_{-\infty}^{0}e^{2\ddelta u}\expect \Bigl\lVert
         X(u+s,\shft_{-s}\point))-X(u+s_0,\shft_{-s_0}\point))
         \Bigr\rVert^2\,du\\
 &+4M^2\int_{-\infty}^{0}e^{2\ddelta u}\expect \Bigl\lVert
         G(u+s,X(u+s_0,\shft_{-s_0}\point))\\
                   &\pushright
                     {-G(u+s_0,X(u+s_0,\shft_{-s_0}\point))
                     \Bigr\rVert^2\un{A_\alpha(u)}\,du}\\
    &+4M^2\int_{-\infty}^{0}e^{2\ddelta u}\expect \Bigl\lVert
         G(u+s,X(u+s_0,\shft_{-s_0}\point))\\
                   &\pushright
                     {-G(u+s_0,X(u+s_0,\shft_{-s_0}\point))
                     \Bigr\rVert^2\un{\Omega\setminus
                     A_\alpha(u)}\,du}.\\
\leq &4M^2\alpha\CCO{{1}+\frac{1}{2\ddelta}}(\ctlip^2+5).
\end{align*}
The result follows since $\alpha$ is arbitrary. 

For $s_0\leq 0$ and $\epsilon>0$, let $\tau$ be an $\epsilon/3$ almost
period of $X$ such that $s_0+\tau>0$, and let $\eta>0$ such that
$s_0+\tau-\eta>0$ and
\begin{equation*}
  \abs{s-s_0}<\eta\Rightarrow
  \expect\norm{\SDE X(s+\tau,\shft_{-s-\tau}\point)-\SDE
    X(s_0+\tau,\shft_{-s_0-\tau}\point)}^2<\frac{\epsilon}{3}. 
\end{equation*}
We have then, for $\abs{s-s_0}<\eta$,
\begin{multline*}
  \expect\norm{\SDE X(s,\shft_{-s}\point)-\SDE
           X(s_0,\shft_{-s_0}\point)}^2\\
\begin{aligned}
 \leq& \expect\norm{\SDE X(s,\shft_{-s}\point)-\SDE
  X(s+\tau,\shft_{-s-\tau}\point)}^2\\
  &+\expect\norm{\SDE X(s+\tau,\shft_{-s-\tau}\point)-\SDE
    X(s_0+\tau,\shft_{-s_0-\tau}\point)}^2\\
  &+\expect\norm{\SDE X(s_0+\tau,\shft_{-s_0-\tau}\point)-\SDE
    X(s_0,\shft_{-s_0}\point)}^2\\
  \leq&\epsilon. 
\end{aligned}
\end{multline*}
\end{proof}

As a byproduct,
we get:
\begin{proposition}\label{prop:periodicstable}
  Assume that $F(\point,x)$ and $G(\point,x)$ are $\tau$-periodic, for
  some $\tau>0$. Then the operator $\SDE$ maps $\PB(\tau)$ into itself. 
\end{proposition}
\begin{proof}
  It suffices to repeat the proof of Proposition \ref{prop:apstable}
  with $\epsilon=0$.
  The result follows since $\alpha$ can be chosen arbitrarily small. 
\end{proof}

\begin{proposition}\label{prop:apAPPD}
  For any $X\in\APB$, the random process $\SDE X$ has a continuous
  modification which is almost periodic in
  path distribution (APPD).  
\end{proposition}
\begin{proof}
  By Theorem \ref{theo:ap-vs-apdist}, $X$ is APFD.   
Furthermore, by \cite[Propoition 7.3]{dapratozabczyk}, $X$ has a 
continuous modification.
To prove that $X$ is APPD,
one can use the method of Da Prato and Tudor \cite{DaPrato-Tudor95},
which is based on Bochner's double sequence criterion.
Let us prove instead Condition \eqref{item:ap-APPD3}
of Theorem \ref{theo:ap-vs-apdist}. In both methods, the key tool is
the convolution inequality \cite[Theorem 6.10]{dapratozabczyk}:
for every $T>0$, 
there exists a constant $\ctconv_T$ such that,
for any predictable random process $Y$
with values in $\HSz$ and $t_0\in\R$ such
that $\expect\CCO{\int_{t_0}^{t_0+T}\norm{Y(u,\point)}^2du}>\infty$,
we have
\[\expect\CCO{\sup_{t\in [t_0,t_0+T]}\norm{\int_{t_0}^t
    S(t-u)Y(u,\point)\,dW(u,\point)}^2}
\leq\ctconv_T\expect\CCO{\int_{t_0}^{t_0+T}\norm{Y(u,\point)}^2du}.\]
Let $J=[t_0,t_0+T]$ be a fixed compact interval.
Let $\epsilon>0$ and let $\alpha>0$.
As in the proof of Proposition \ref{prop:apstable}, we can find
$\eta>0$ and a compact subset $K$ of $\espH$ such that, for all
$t\in\R$ and all $A\in\tribu$,
\begin{gather*}
  \prob\CCO{\accol{X(t,\point)\not\in K}}
                      \leq \eta\leq \frac{\epsilon^2}{\alpha^2},\\
  \prob(A)\leq\eta\Rightarrow
  \expect\CCO{\bigl(1+\norm{X(t,\point)}^2\bigr)\un{A}}
                   \leq \frac{\epsilon^2}{\alpha^2}.
\end{gather*}
Let $\tau$ be a common $(\epsilon/\alpha)$-almost period of
$t\mapsto X(t,\shft_t\point)$, $t\mapsto F(t,\point,x)$ and
$t\mapsto G(t,\point,x)$ for all $x\in K$. 
We have for every $t\geq t_0$,
using the calculation of \eqref{eq:changeofvar}, 
\begin{align*}
  \transl_{\tau}X(t,\point)
  &=S(t-t_0)\transl_{\tau}X(t_0,\point)
    +\int_{t_0}^tS(t-u)F(u+\tau,X(u+\tau,\shft_{-\tau}\point))\,du\\
    & +\int_{-\infty}^{t}
     S(t-u)G(u+\tau,X(u+\tau,\shft_{-\tau}\point))\,d
      W(u,\point).
\end{align*}
We deduce
\begin{align*}
  \expect&\CCO{\sup_{t\in J}\norm{\transl_{\tau}X(t)-X(t)}^2}\\
  \leq&
    3\expect\CCO{\sup_{t\in J}\norm{
      S(t-t_0)\bigl(\transl_{\tau}X(t_0,\point)-X(t_0,\point)\bigr)}^2}\\
  &+3\expect\CCO{\sup_{t\in J}\norm{
    \int_{t_0}^tS(t-u)
    \bigl(F(u+\tau,X(u+\tau,\shft_{-\tau}\point))
    -(F(u,X(u,\point))
    \bigr)\,du}^2} \\
         &+3\expect\CCO{\sup_{t\in J}\norm{
   \int_{t_0}^tS(t-u)
    \bigl(G(u+\tau,X(u+\tau,\shft_{-\tau}\point))
    -(G(u,X(u,\point))
    \bigr)\,dW(u,\point)
           }^2}\\
  &=:3A+3B+3C.
\end{align*}
We have
\begin{equation*}
  A\leq \expect\CCO{\norm{
      \transl_{\tau}X(t_0,\point)-X(t_0,\point)}^2}
    \leq \frac{\epsilon^2}{\alpha^2}.
\end{equation*}
Furthermore, let us denote, for $u\in J$,
\[\Omega_u=\{ X(u+\tau,\shft_{-\tau}\point)\in K \text{ and } X(u,\point)\in
  K\}.\]
We have
\begin{align*}
B  \leq&2M^2T\expect\CCO{
\int_{t_0}^{t_0+T}
    \norm{F(u+\tau,X(u+\tau,\shft_{-\tau}\point))
    -(F(u+\tau,X(u,\point))
    }^2\,du
        }\\
     &+2M^2T\expect\CCO{
\int_{t_0}^{t_0+T}
    \norm{F(u+\tau,X(u,\point))
    -(F(u,X(u,\point))
    }^2\,du
        }\\
  \leq &2M^2T\ctlip^2
         \int_{t_0}^{t_0+T}\expect\CCO{
    \norm{X(u+\tau,\shft_{-\tau}\point)
    -X(u,\point)
    }^2}\,du
         \\
&+2M^2T\int_{t_0}^{t_0+T}\expect\CCO{
    \norm{F(u+\tau,X(u,\point))
    -(F(u,X(u,\point))
    }^2\un{\Omega_u}}\,du\\
  &+2M^2T\int_{t_0}^{t_0+T}\expect\CCO{
    \norm{F(u+\tau,X(u,\point))
    -(F(u,X(u,\point))
    }^2\un{\Omega\setminus \Omega_u}}\,du\\
  \leq & 2M^2T\CCO{\ctlip^2T\frac{\epsilon^2}{\alpha^2}+T\frac{\epsilon^2}{\alpha^2}
         +2\ctgrowth^2\int_{t_0}^{t_0+T}\expect\CCO{1+\norm{X(u,\point) }}^2
         \un{\Omega\setminus \Omega_u}du}\\
  \leq&\frac{2M^2T^2\epsilon^2}{\alpha^2}(\ctlip^2+1+2\ctgrowth^2T). 
\end{align*}
Similarly, we get, using the convolution inequality,
\begin{align*}
  C\leq&\ctconv_TM^2\expect\CCO{\int_{t_0}^{t_0+T}\norm{
G(u+\tau,X(u+\tau,\shft_{-\tau}\point))
    -(G(u,X(u,\point))
         }^2du}\\
  \leq&\ctconv_T\frac{2M^2T\epsilon^2}{\alpha^2}(\ctlip^2+1+2\ctgrowth^2T). 
\end{align*}
We deduce that, for some constant $\kappa>0$,
\begin{equation*}
   \expect\CCO{\sup_{t\in J}\norm{\transl_{\tau}X(t)-X(t)}^2}
  \leq3(A+B+C)
        \leq \kappa\frac{\epsilon^2}{\alpha^2},
\end{equation*}
thus, taking $\alpha=\sqrt{\kappa}$, we obtain that $\tau$ is an
$\epsilon$-period of the map $Z_J$ defined in \eqref{eq:Z_J}.
We deduce that the set of $\epsilon$-almost periods of $Z_J$ is
relatively dense. 
By Theorem \ref{theo:ap-vs-apdist}, this shows that $X$ is APPD. 
\end{proof}

The following result is given in a slightly different setting in
\cite[Theorem 3.1]{KMRF}: 
\begin{lemma}\label{lem:contraction}
 Assume that Hypotheses (H\ref{hyp:C0}) and (H\ref{hyp:FG})
are satisfied, with $M=1$, and that
\begin{equation}\label{eq:theta}
  2\ctlip^{2}\Bigl({1}+\frac{1}{2\ddelta}\Bigr)<1.
\end{equation}
Then the operator $SDE$ is a contraction in $\CUB$. 
\end{lemma}
\begin{proof}
  Let $X,Y\in\CUB$.
  We have, for any $t\in\R$, 
\begin{multline*}
 \expect\lVert \SDE X(t,\point) -  \SDE Y(t,\point)\rVert^2\\
\begin{aligned}
\leq& 2 \expect\Bigl(\int_{-\infty}^{t}S(t-s)
    \lVert F(s, X(s,\point))- F(s,Y(s,\point))\rVert \,ds\Bigl)^2\\
  & + 2\expect
  \Bigl(\| \int_{-\infty}^{t}S(t-s)
  \bigl(G(s, X(s,\point)) -
               G(s,Y(s,\point))\bigr)\,dW(s,\point)\|\Bigl)^2\\
\leq & 2\ddelta
  \int_{-\infty}^{t} e^{-\ddelta(t-s)}\expect
  \lVert F(s, X(s,\point))- F(s,Y(s,\point))\rVert^2 \,ds\\
  &+2\int_{-\infty}^{t} e^{-2\ddelta(t-s)}\expect
  \lVert G(s, X(s,\point))- G(s,Y(s,\point))\rVert^2 \,ds\\
= &2\ctlip^{2}\Bigl({1}+\frac{1}{2\ddelta}\Bigr)
\sup_{s\in\R}\expect
  \lVert X(s,\point)- Y(s,\point)\rVert^2.
\end{aligned}
\end{multline*}
\end{proof}

The following result
is similar to \cite[Theorems 4.1 and
4.3]{dapratozabczyk}, with different hypotheses.
It improves \cite[Theorem 3.1]{KMRF} on almost
periodic solutions, notably since $F$  and $G$ are assumed to be almost
periodic uniformly with respect to compact sets, instead of bounded sets. 
The main novelty is the addition of $\shft$-almost
periodicity, which  allows for a different proof of almost
periodicity in path distribution. 
This method provides a result on periodic solutions as a particular case.

\begin{theorem}[Almost periodic or periodic 
  solution]\label{theo:apsolution}
Assume that Hypotheses (H\ref{hyp:C0})-(H\ref{hyp:FG})-(H\ref{hyp:pp})
are satisfied, with $M=1$, and that \eqref{eq:theta} holds true. 
Then \eqref{eq:EDSfond} has a unique bounded mild solution $X$
which satisfies \eqref{eq:mild} and has a continuous modification,
and $X$ is $\shft$-almost periodic and almost periodic in path
distribution (APPD).

Furthermore, if
$F(\point,x)$ and $G(\point,x)$ are $\tau$-periodic, for
  some $\tau>0$ and all $x\in\espH$, 
then $X$ is
$\shft$-$\tau$-periodic, thus it is periodic in finite dimensional
distribution (PFD). 
In particular, if $F$ and $G$ do not depend on the first variable,
$X$ is $\shft$-stationary. 
\end{theorem}
\begin{proof}
  By 
  Proposition \ref{prop:apstable}
  or Proposition \ref{prop:periodicstable},
  and Proposition \ref{prop:closure}, 
  to prove the existence and 
  uniqueness of a $\shft$-almost periodic
  solution
  (or  a $\shft$-$\tau$-periodic solution if
$F(\point,x)$ and $G(\point,x)$ are $\tau$-periodic, for
  some $\tau>0$ and all $x\in\espH$),
  it is sufficient to show that the
  operator $\SDE$ has an attracting fixed point in $\CUB$,
but this is a consequence of Lemma \ref{lem:contraction}. 
By Proposition \ref{prop:apAPPD}, $X$ is APPD.
If $F$ and $G$ are $\tau$-periodic with respect to
the first variable, $X$ is PFD by Theorem \ref{theo:ap-vs-apdist}.
\end{proof}


\def\cprime{$'$}


\end{document}